\documentclass{article}
\usepackage{}
\usepackage{amsfonts}
\usepackage{amsthm,amssymb,amsmath,amscd}
\usepackage{wasysym}
\usepackage[all]{xy}
\usepackage{mathrsfs}
\usepackage{multirow}
\usepackage{txfonts}
\usepackage{color}
\usepackage{url}

\usepackage{lscape}

\DeclareMathAlphabet{\mathcal}{OMS}{cmsy}{m}{n}
\DeclareSymbolFont{largesymbols}{OMX}{cmex}{m}{n}

\objectmargin{1pt}

\newtheorem{Def}{Definition}[section]
\newtheorem{Prop}[Def]{Proposition}
\newtheorem{Theo}[Def]{Theorem}
\newtheorem{Lem}[Def]{Lemma}
\newtheorem{Koro}[Def]{Corollary}

\newcommand{\defCategory}[2]{
  \newcommand{#1}{#2\defvariable}
  }

\newcommand{\defvariable}[2][]{
\if\relax\detokenize{#1}\relax  
   \if\relax\detokenize{#2}\relax
    \else
    \left({#2}\right)
    \fi
\else
  ^{{#1}}\left({#2}\right)
\fi
 }

\def\cpx#1{{#1}^{\bullet}}
\def\subcpx#1#2{\sigma_{#1}#2}
 \def\con{{\rm con}}

 \defCategory{\C}{\mathscr{C}}
 \defCategory{\K}{\mathscr{K}}
 \defCategory{\D}{\mathscr{D}}

 \def\Kb#1{\K[{\rm b}]{#1}}
 \def\Kf#1{\K[-]{#1}}
 \def\Kz#1{\K[+]{#1}}

 \def\Db#1{\D[{\rm b}]{#1}}

\def\opp{^{\rm op}}

 \DeclareMathOperator{\add}{add}

 \DeclareMathOperator{\Ext}{Ext}
 \DeclareMathOperator{\findim}{fin.\!dim}
 \DeclareMathOperator{\gldim}{gl.\!dim}
 
 \DeclareMathOperator{\Hom}{Hom}
 
 \DeclareMathOperator{\HomP}{Hom^{\bullet}}
 \DeclareMathOperator{\Img}{Im}

 \DeclareMathOperator{\projdim}{proj.dim}

 \DeclareMathOperator{\stHom}{\underline{Hom}}
 \DeclareMathOperator{\otimesL}{\stackrel{\mathbf{L}}{\otimes}}
 
 \DeclareMathOperator{\thick}{thick}

\def\Hom{{\rm Hom}}
\def\Ext{{\rm Ext}}

\def\HomP{\Hom^{\bullet}}
\def\stHom{\underline{\Hom}}
\def\otimesL{\stackrel{\bf L}{\otimes}}

\def\Modcat#1{#1\mbox{-}{\rm Mod}}
\def\modcat#1{#1\mbox{-}{\rm mod}}
\def\stModcat#1{#1\mbox{-}\underline{{\rm Mod}}}
\def\stmodcat#1{#1\mbox{-}\underline{{\rm mod}}}
\def\pModcat#1{#1\mbox{-}{\rm Proj}}
\def\pmodcat#1{#1\mbox{-}{\rm proj}}
\def\GProj#1{#1\mbox{-}{\cal GP}}

\def\Gproj#1{#1\mbox{-}{\it f\cal GP}}
\def\stGProj#1{#1\mbox{-}{\underline{{\cal GP}}}}
\def\stGproj#1{#1\mbox{-}\underline{{\it f\cal GP}}}

\newcommand{\perpg}[1]{^{\perp_{>{#1}}}}

\newcommand{\lra}{\longrightarrow}
\newcommand{\lla}{\longleftarrow}
\newcommand{\ra}{\rightarrow}

\newcommand{\lraf}[1]{\stackrel{#1}{\lra}}
\newcommand{\llaf}[1]{\stackrel{#1}{\lla}}
\newcommand{\raf}[1]{\stackrel{#1}{\ra}}

\newcommand{\proja}{\mathcal{P}_{\!\!\!\scriptscriptstyle{\mathcal{A}}}}
\newcommand{\projb}{\mathcal{P}_{\!\!\scriptscriptstyle{\mathcal{B}}}}
\newcommand{\projc}{\mathcal{P}_{\!\!\scriptscriptstyle{\mathcal{C}}}}

\newcommand{\sta}{\underline{\mathcal{A}}}
\newcommand{\stb}{\underline{\mathcal{B}}}

\title{Stable functors of derived equivalences and Gorenstein projective modules}
\author{\sc WEI HU and SHENGYONG PAN}
\date{}

\begin{document}
\maketitle

\begin{abstract}

From certain triangle functors, called non-negative functors,  between the bounded derived categories of abelian categories with enough projective objects, we introduce their stable functors which are certain additive functors between the stable categories of the abelian categories.  The construction generalizes a previous work by Hu and Xi. We show that the stable functors of non-negative functors have nice exactness property and are compatible with composition of functors.   This allows us to compare conveniently  the homological properties of objects linked by the stable functors. Particularly, we prove that   the stable functor of a derived equivalence between two arbitrary rings provides an explicit triangle equivalence between the stable categories of Gorenstein projective modules.  This generalizes a result of Y. Kato.  Our results can also be applied to  provide shorter proofs of some known results on homological conjectures.
\end{abstract}


\section{Introduction}
Derived equivalences were introduced by Grothendieck and Verdier in 1960s, and play an important role nowadays  in many branches of mathematics and physics, especially in representation theory and in algebraic geometry. A derived equivalence is a triangle equivalence between the derived categories of complexes over certain abelian categories such as the module category of a ring or the category of coherent sheaves over some variety.  For derived equivalent abelian categories, it is very hard to directly compare the objects in the given abelian categories, since a derived equivalence typically takes objects in one abelian category  to complexes over the other.  

\smallskip
For an arbitrary derived equivalence $F$ between two  Artin algebras,  a functor $\bar{F}$ between the stable module categories were introduced in \cite{Hu2010}, called the stable functor of $F$.  This functor allows us to compare the modules over one algebra with the modules over the other.  Another nice property of this functor is that $\bar{F}$ is a stable equivalence of Morita type in case that $F$ is an almost $\nu$-stable standard derived equivalence, This generalizes a classic result \cite{Rickard1991} of Rickard which says that a derived equivalence between two selfinjective algebras always induces a stable equivalence of Morita type. However, in \cite{Hu2010}, many basic questions on the stable functor remain. For instance, we even don't know whether the stable functor is uniquely determined by the given derived equivalence, and whether the definition of the stable functor is compatible with composition of derived equivalences. 

\smallskip
In this paper, we shall look for a more general and systematical definition of stable functors, and generalize the notion of stable functors in two directions. One direction is that,  instead of module categories of Artin algebras, we consider arbitrary abelian categories with enough projective objects.  The other direction is that, instead of derived equivalences, we consider certain triangle functors, called non-negative functors,  between the derived categories. Note that this condition is not restrictive: all derived equivalences between rings are non-negative up to shifts.  We shall prove, in this general framework, that the stable functor is uniquely determined by the given non-negative functor (Theorem \ref{theorem-unique-F1F2}) and is compatible with the composition of non-negative functors (Theorem \ref{theorem-stFunctor-composite}).

\smallskip 
Our theory of stable functors can be applied to study stable categories of Gorenstein projective modules of derived equivalent rings, namely, the stable functor of a derived equivalence between two arbitrary rings provides an explicit triangle equivalence between their stable categories of Gorenstein projective modules (Corollary \ref{corollary-stgproj-equiv-ring-1}). 
 Gorenstein projective modules go back to a work of Auslander and Bridger \cite{Auslander1969}. Since then they have attracted more attention and have also nice applications in commutative algebra, algebraic geometry, singularity theory and relative homological algebra. In general, the size and homological complexity of the  stable category of Gorenstein projective modules measure how far the ring is from being Gorenstein. A nice feature of the  stable category of Gorenstein projective modules is that it is a triangulated category, and admits a full triangulated embedding into the singularity category in Orlov's sense, which is an equivalence if and only if the ring is Goresntein.

\smallskip
This paper is organized as follows. In Section \ref{section-preliminaries}  we recall some basic definitions and facts required in proofs.  Section \ref{section-invariant-space} is devoted to  studying for which complexes the localization functor from the homotopy category to the derived category preserves homomorphism spaces. The theory of stable functors will be given in  Section \ref{section-stablefunctor}, and will be applied to study stable category of Gorenstein projective modules in Section \ref{section-Gorenstein}.  An example is given in Section \ref{section-example} to illustrate how we can compute the Gorenstein projective modules over an algebra via the stable functor. Finally, we stress in Section \ref{section-remarks} that our results can be used to give shorter proofs of some known results on homological conjectures.

%
%

\section{Preliminaries}\label{section-preliminaries}
In this section, we recall some basic definitions and collect some basic facts for later use.

\smallskip
Throughout this paper, unless specified otherwise, all categories are additive categories, and all functors are additive functors.  The composite of two morphisms $f: X\ra Y$ and $g: Y\ra Z$ in a category ${\cal C}$ will be denoted by $fg$. If $f: X\ra Y$ is a map between two sets, then the image of an element  $x\in X$ will be denoted by $(x)f$. However, we will deal with functors in a different manner. The composite of two functors $F: {\cal C}\ra {\cal D}$ and $G: {\cal D}\ra {\cal E}$ will be denoted by $GF$. For each object $X$ in ${\cal C}$, we write $F(X)$ for the corresponding object in ${\cal D}$, and for each morphism $f: X\ra Y$ in ${\cal C}$ we write $F(f)$ for the corresponding morphism in ${\cal D}$ from $F(X)$ to $F(Y)$. For an object $M$ in an additive category ${\cal C}$, we use $\add(M)$ to denote the full subcategory of ${\cal C}$ consisting of direct summands of finite direct sums of copies of $M$.

Let $\mathcal{A}$ be an additive  category. A complex $\cpx{X}$ over ${\cal A}$ is a sequences $d_X^i$ between objects $X^i$ in ${\cal A}$: $\cdots\lra X^{i-1}\lraf{d_X^{i-1}}X^i\lraf{d_X^i}X^{i+1}\lraf{d_X^{i+1}}\cdots$ such that $d_X^id_X^{i+1}=0$ for all $i\in\mathbb{Z}$.  The category of complexes over ${\cal A}$, in which morphisms are chain maps, is denoted by $\C{\cal A}$, and the corresponding homotopy category is denoted by $\K{\cal A}$. When ${\cal A}$ is an abelian category, we write $\D{\cal A}$ for the derived category of ${\cal A}$.
We also write $\Kb{\cal A}$, $\Kf{\cal A}$ and $\Kz{\cal A}$ for the full subcategories of $\K{\cal A}$ consisting of complexes isomorphic to bounded complexes, complexes bounded above, and complexes bounded below, respectively.  Similarly, for $*\in\{b, -, +\}$, we have $\D[*]{\cal A}$. Moreover, for integers $m\leq n$ and for a collection of objects ${\cal  X}$, we write $\D[{[m, n]}]{\cal X}$ for the full subcategory of $\D{\cal A}$ consisting of complexes $\cpx{X}$ isomorphic in $\D{\mathcal{A}}$ to complexes  with terms in ${\cal X}$ of the form
$$0\lra X^m\lra\cdots\lra X^n\lra 0.$$
For each complex $\cpx{X}$ over $\mathcal{A}$, its $i$th cohomology is denoted by $H^i(\cpx{X})$.

\medskip
The homotopy category of an additive category, and the derived category of an abelian category are both triangulated categories.  For basic facts on triangulated categories, we refer to Neeman's book \cite{Neeman2001}.  However, the shift functor of a triangulated category will be denoted by $[1]$ in this paper.  In the homotopy category, or the derived category of an abelian category, the shift functor acts on a complex by moving the complex to the left by one degree, and changing the sign of the differentials.

\medskip
Suppose that $\mathcal{A}$ is an abelian category. There is a full embedding ${\cal A}\hookrightarrow \D{\cal A}$ by viewing an object in ${\cal A}$ as a complex in $\D{\cal A}$ concentrated in degree zero.  Let $\mathscr{X}$ be a collection of objects in $\D{\cal A}$ and let $n$ be an integer.  We define a full subcategory of $\D{\mathcal{A}}$:
$${}^{\perp_{> n}}\mathscr{X}:=\{\cpx{Z}\in\D{\mathcal{A}}\, |\, \Hom_{\D{\cal A}}(\cpx{Z}, \cpx{X}[i])=0\mbox{ for all } i>n\mbox{ and for all }\cpx{X}\in\mathscr{X}\},$$
For simplicity, we write ${}^{\perp}\!\mathscr{X}$ for ${}^{\perp_{>0}}\mathscr{X}$.

\medskip
Suppose that $\mathcal{A}$ is an abelian category with enough projective objects.  Let $\proja$ be the full subcategory of $\mathcal{A}$ consisting of all projective objects. The stable category of $\mathcal{A}$, denoted by $\underline{\mathcal{A}}$,   is defined to be the additive quotient $\mathcal{A}/\proja$, where the objects are the same as those in $\mathcal{A}$ and the morphism space $\Hom_{\underline{\mathcal{A}}}(X, Y)$ is the quotient space of $\Hom_{\mathcal{A}}(X, Y)$ modulo all morphisms factorizing through projective objects.
Two objects $X$ and $Y$ are isomorphic in $\underline{\mathcal{A}}$ if and only if there are projective objects $P$ and $Q$ such that $X\oplus Q\simeq Y\oplus P$ in $\mathcal{A}$. This seems not so obvious.
Indeed, first of all, it is easy to check that the injection $X\ra X\oplus Q$ is an isomorphism in $\underline{\mathcal{A}}$. So, if $X\oplus Q\simeq Y\oplus P$ in ${\mathcal{A}}$ with $P, Q$ projective, then $X$ and $Y$ are isomorphic in $\underline{\mathcal{A}}$. Conversely, suppose that $f: X\ra Y$ is a morphism in $\mathcal{A}$ such that its image $\underline{f}: X\ra Y$ in $\Hom_{\underline{\mathcal{A}}}(X, Y)$ is an isomorphism.  Then there is a morphism $g: Y\ra X$ such that $1_X-fg$ factorizes through some projective object $P$. Namely, there exist morphisms $\alpha: X\ra P$ and $\beta: P\ra X$ such that $1_X=fg+\alpha\beta$. Then we can form a split exact sequence
$$0\lra X\lraf{[f, \alpha]}Y\oplus P\lraf{\left[\begin{smallmatrix} u\\v \end{smallmatrix}\right]}Q\lra 0.$$
It follows that $fu=-\alpha v$ factorizes through the projective object $P$. This implies that $\underline{f}\underline{u}=0$. However, the morphism $\underline{f}$ is an isomorphism. Hence $\underline{u}=0$, and therefore $u$ factorizes through a projective object $P'$, say, $u=ab$ for some morphisms $a: Y\ra P'$ and $b: P'\ra Q$. Thus $\left[\begin{smallmatrix} u\\v \end{smallmatrix}\right]$ factorizes through the morphism $P'\oplus P\lraf{\left[\begin{smallmatrix} b\\v \end{smallmatrix}\right]}Q$.  The above split exact sequence indicates that $1_Q$ factorizes through $\left[\begin{smallmatrix} u\\v \end{smallmatrix}\right]$, and consequently factorizes through $\left[\begin{smallmatrix} b\\v \end{smallmatrix}\right]$. Hence $Q$ is isomorphic to a direct summand of $P'\oplus P$ and has to be projective.  This establishes that $X\oplus Q\simeq Y\oplus P$ with $P, Q$ projective.

 \medskip
 Let $A$ be an arbitrary ring with identity. The category $\Modcat{A}$ of unitary left  $A$-modules is an abelian category with enough projective objects.  We use $\modcat{A}$ to denote the full subcategory of $\Modcat{A}$ consisting of  finitely presented $A$-modules, that is, $A$-modules $X$ admitting a projective presentation $P_1\ra P_0\lra X\ra 0$ with $P_i$ finitely generated projective for $i=0, 1$.  The category $\modcat{A}$ is abelian when $A$ is left coherent. The full subcategory of $\Modcat{A}$ consisting of all projective modules is denoted by $\pModcat{A}$, and the category of finitely generated $A$-modules is written as $\pmodcat{A}$. Note that $\pmodcat{A}$ are precisely those projective modules in $\modcat{A}$. The stable category of $\Modcat{A}$ is denoted by $\stModcat{A}$, in which morphism space is denoted by $\stHom_A(X, Y)$ for each pair of $A$-modules $X$ and $Y$.  For a full subcategory $\mathscr{X}$ of $\Modcat{A}$, we denote by $\underline{\mathscr{X}}$ the full subcategory of $\stModcat{A}$ consisting of all modules in $\mathscr{X}$. However, the full subcategory of $\stModcat{A}$ consisting of finitely presented modules is denoted by $\stmodcat{A}$

\medskip
Two rings $A$ and $B$ are said to be {\em derived equivalent} if the following equivalent conditions are satisfied.

\smallskip
(1).  $\D{\Modcat{A}}$ and $\D{\Modcat{B}}$ are equivalent as triangulated categories.

(2). $\Db{\Modcat{A}}$ and $\Db{\Modcat{B}}$ are equivalent as triangulated categories.

(3). $\Kb{\pModcat{A}}$ and $\Kb{\pModcat{B}}$ are equivalent as triangulated categories.

(4). $\Kb{\pmodcat{A}}$ and $\Kb{\pmodcat{B}}$ are equivalent as triangulated categories.

(5). There is a complex $\cpx{T}$ in $\Kb{\pmodcat{A}}$
satisfying the conditions:

  \quad \quad (a). $\Hom_{\Kb{\pmodcat{A}}}(\cpx{T},\cpx{T}[n])= 0$ for all $n\neq 0$,

   \quad \quad (b). $\add(\cpx{T})$ generates $\Kb{\pmodcat{A}}$ as a triangulated category,

\quad\quad  such that the endomorphism algebra of $\cpx{T}$ in $\Kb{\pmodcat{A}}$ is isomorphic to $B$.

\medskip
{\parindent-0pt For the } proof that the above conditions are indeed equivalent, we refer to \cite{Rickard1989a,Keller1994}. If the algebras $A$ and $B$ are left coherent, then the above equivalent conditions are further equivalent to the following condition.

\medskip
(6). $\Db{\modcat{A}}$ and $\Db{\modcat{B}}$ are equivalent as triangulated categories.

\medskip
{\parindent=0pt   A} complex $\cpx{T}$ satisfying the conditions (a) and (b) above is called a {\em tilting complex}. A triangle equivalence functor $F: \Db{\Modcat{A}}\ra \Db{\Modcat{B}}$ is called a {\em derived equivalence}. In this case, the image $F(A)$ is isomorphic in $\Db{\Modcat{B}}$ to a tilting complex, and there is a tilting complex $\cpx{T}$ over $A$ such that $F(\cpx{T})$ is isomorphic to $B$ in $\Db{\Modcat{B}}$.  The complex $\cpx{T}$ is called an {\em associated tilting complex } of $F$.  The following is an easy lemma for the associated tilting complexes.  For the convenience of the reader, we provide a proof.

\begin{Lem}
Let $A$ and $B$ be two rings, and let $F: \Db{\Modcat{A}}\lra \Db{\Modcat{B}}$ be a derived equivalence. Then $F(A)$ is isomorphic in $\Db{\modcat{B}}$ to a complex $\cpx{\bar{T}}\in\Kb{\pmodcat{B}}$  of the form
$$0\lra \bar{T}^0\lra \bar{T}^1\lra\cdots\lra\bar{T}^n\lra 0$$
for some $n\geq 0$ if and only if $F^{-1}(B)$ is isomorphic  in $\Db{\Modcat{A}}$ to a complex $\cpx{T}\in\Kb{\pmodcat{A}}$ of the form
$$0\lra T^{-n}\lra\cdots\lra T^{-1}\lra T^0\lra 0.$$
\label{lemma-tiltCompForm}
\end{Lem}
\begin{proof}
We prove the necessity, the proof of the sufficiency is similar. Suppose that $F(A)$ is isomorphic to a complex $\cpx{\bar{T}}$ in $\Kb{\pmodcat{B}}$ of the form
$$0\lra \bar{T}^0\lra \bar{T}^1\lra\cdots\lra\bar{T}^n\lra 0, $$
and $\cpx{T}$ is a complex in $\Kb{\pmodcat{A}}$ such that $F(\cpx{T})\simeq B$.  Then
$$\Hom_{\Db{\Modcat{A}}}(A, \cpx{T}[i])\simeq\Hom_{\Db{\Modcat{B}}}(\cpx{\bar{T}}, B[i])=0$$
for all $i>0$. Hence $\cpx{T}$ has zero homology in all positive degrees. Since all the terms of $\cpx{T}$ are projective, the complex $\cpx{T}$ is split in all positive degrees, and is isomorphic in $\Kb{\pmodcat{A}}$ to a complex with zero terms in all positive degrees. Thus, we can assume that $T^i=0$ for all $i>0$.  To prove that $\cpx{T}$ is isomorphic to a complex in $\Kb{\pmodcat{A}}$ with zero terms in all degrees $<-n$, it suffices to show that $\Hom_{\Db{\Modcat{A}}}(\cpx{T}, P[i])=0$ for all $i>n$ and for all finitely generated projective $A$-module $P$. Actually,  since $F(P)$ is in $\add(\cpx{\bar{T}})$, we can deduce that
$$\Hom_{\Db{\Modcat{A}}}(\cpx{T}, P[i])\simeq \Hom_{\Db{\Modcat{B}}}(B, F(P)[i])=0$$
for all $i>n$.
\end{proof}

 \section{Homomorphism spaces invariant from $\K{{\cal A}}$ to $\D{{\cal A}}$}\label{section-invariant-space}

 Let  ${\cal A}$ be an abelian category, let $q: \K{{\cal A}}\lra\D{{\cal A}}$ be the localization functor.  The morphisms in the derived category are ``complicated", while the morphisms in the homotopy category are relatively ``simple": they can be presented by chain maps. It is very natural to ask the following question:

 \medskip
 {\em For which complexes $\cpx{X}$ and $\cpx{Y}$, the induced map
 $$q_{(\cpx{X},\cpx{Y})}: \Hom_{\K{{\cal A}}}(\cpx{X},\cpx{Y})\lra \Hom_{\D{{\cal A}}}(\cpx{X},\cpx{Y})$$

is an isomorphism? }

\medskip
{\parindent=0pt It} is known that this is true in case that $\cpx{X}$ is an above-bounded complex of projective objects, or $\cpx{Y}$ is a below-bounded complex of injective objects.   In this section, we shall prove the following very useful proposition, which allows us to get morphisms between objects from morphisms between complexes in the derived category. It seems that this has not appeared elsewhere in the literature.

 \begin{Prop}
 Let ${\cal A}$ be an abelian category, and let $\cpx{X}$ and $\cpx{Y}$  be above-bounded and below-bounded complexes of objects in  ${\cal A}$, respectively.   Suppose that $X^i\in{}^{\perp}Y^j$ for all integers $j<i$. Then
the induced map
$$q_{(\cpx{X},\cpx{Y}[n])}: \Hom_{\K{{\cal A}}}(\cpx{X}, \cpx{Y}[n])\lra \Hom_{\D{{\cal A}}}(\cpx{X}, \cpx{Y}[n])$$
 is an isomorphism for all $n\leq 0$, and is a monomorphism for $n=1$.
\label{prop-cpxXcpxY}
\end{Prop}

This proposition generalizes \cite[Lemma 2.2]{Hu2010}, and its proof will be given after several lemmas.

\medskip
 Let $F: {\cal T}\lra {\cal S}$ be a triangle functor between two triangulated categories, and let $M\in {\cal T}$ be an object. We define $\mathscr{U}^F_M$ to be the full subcategory of ${\cal T}$ consisting of objects $X$ satisfying the following two conditions.

 \smallskip
 (1)  $F_{(X, M[i])}: \Hom_{\cal T}(X, M[i])\lra \Hom_{\cal S}(F(X), F(M)[i])$  is an isomorphism for all $i\leq 0$.

 (2) $F_{(X, M[1])}: \Hom_{\cal T}(X, M[1])\lra \Hom_{\cal S}(F(X), F(M)[1])$ is monic.

 \medskip
 {\parindent=0pt Let} ${\cal T}$ be a triangulated category, and let $\mathscr{X}$ and $\mathscr{Y}$ be full subcategories of ${\cal T}$. We define
 $$\mathscr{X}*\mathscr{Y}:=\{Z\in {\cal T}\,|\,\mbox{There is a triangle }X\ra Z\ra Y\ra X[1]\mbox{ with }X\in\mathscr{X}\mbox{ and }Y\in\mathscr{Y}\}$$
 It is well known that ``$*$" is associative, that is,  $(\mathscr{X}*\mathscr{Y})*\mathscr{Z}=\mathscr{X}*(\mathscr{Y}*\mathscr{Z})$ for any full subcategories $\mathscr{X}, \mathscr{Y}$ and $\mathscr{Z}$ of ${\cal T}$.  So, for full subcategories $\mathscr{X}_1, \cdots, \mathscr{X}_n$ of ${\cal T}$, we can simply write $\mathscr{X}_1*\cdots *\mathscr{X}_n$.

 \begin{Lem}
  Let $F: {\cal T}\lra {\cal S}$ be a triangle functor between triangulated categories ${\cal T}$ and ${\cal S}$. Then we have the following.

 $(1)$. Suppose that $M\in {\cal T}$, and $\mathscr{X}_i\subseteq \mathscr{U}^F_M$ for $i=1, \cdots, n$. Then $\mathscr{X}_1*\cdots *\mathscr{X}_n\subseteq \mathscr{U}^F_{M}$.

 $(2)$.  Suppose that $M_i\in {\cal T}$, and $X\in\mathscr{U}^F_{M_i}$ for $i=1, \cdots, n$. Then $X\in \mathscr{U}^F_M$ for all $M\in {\{M_1\}*\cdots *\{M_n\}}$.
    \label{lemma-x1x2+m1m2}
 \end{Lem}

 \begin{proof}
 (1). Clearly, we only need to prove the case that $n=2$. Let $X$ be an object in $\mathscr{X}_1*\mathscr{X}_2$. There is a triangle $X_1\ra X\ra X_2\ra X_1[1]$ in ${\cal T}$ with $X_i\in\mathscr{X}_i$ for $i=1, 2$.  For simplicity, we write ${\cal T}(-, -)$ for $\Hom_{\cal T}(-, -)$. Then, for each integer $i$, we can form a commutative diagram with exact rows.
 $$\xymatrix{
 {\cal T}(X_1, M[i-1]) \ar[r]\ar[d]_{F_{(X_1, M[i-1])}} &{\cal T}(X_2, M[i])\ar[r]\ar[d]_{F_{(X_2, M[i])}} &{\cal T}(X, M[i])\ar[r]\ar[d]_{F_{(X, M[i])}} &{\cal T}(X_1, M[i])\ar[r]\ar[d]_{F_{(X_1, M[i])}} &{\cal T}(X_2, M[i+1])\ar[d]_{F_{(X_2, M[i+1])}}\\
 {\cal S}\big(FX_1, FM[i-1]\big) \ar[r] &{\cal S}(FX_2, FM[i])\ar[r] &{\cal S}(FX, FM[i])\ar[r] &{\cal S}(FX_1, FM[i])\ar[r] &{\cal S}(FX_2, FM[i+1])\\
 }$$
 If $i\leq 0$, then, by assumption, the maps  $F_{(X_1, M[i-1])}, F_{(X_2, M[i])}, F_{(X_1, M[i])}$ are isomorphisms and  $F_{(X_2, M[n+1])}$ is monic. By Five Lemma,  the map $F_{(X, M[i])}$ is an isomorphism in this case. Our assumption also indicates that $F_{(X_2, M[1])}$ and $F_{(X_1, M[1])}$ are monic, and  $F_{(X_1, M)}$ is an isomorphism.  By Five Lemma again, the map $F_{(X, M[1])}$ is monic. Hence $X\in\mathscr{U}^F_M$.  The proof of (2) is similar to that of (1). We leave it to the reader.
  \end{proof}

Let $X$ and $Y$ be two objects in an abelian category ${\cal A}$, and let $q: \K{\cal A}\ra \D{\cal A}$ be the localization functor.  Then it is straightforward to check that $X[i]\in\mathscr{U}^q_{Y[j]}$ for all $i\geq j$. If $Y\in X^{\perp}$, then $X[i]\in \mathscr{U}^q_{Y[j]}$ for all integers $i$ and $j$, since $q_{(X, Y[m])}: \Hom_{\K{\cal A}}(X, Y[m])\ra\Hom_{\D{\cal A}}(X, Y[m])$ is an isomorphism for all integers $m$ in this case.  If $\cpx{Y}$ is a complex with $Y^i=0$ for all $i<n$, then $X[i]\in\mathscr{U}^q_{\cpx{Y}}$ for all $i\geq -n+2$.  In this case $\Hom_{\K{\cal A}}(X[i], \cpx{Y}[m])=\Hom_{\D{\cal A}}(X[i], \cpx{Y}[m])=0$ for all $m\leq 1$.  Keeping these basic facts in mind helps us to prove the following lemma.

\begin{Lem}
Let ${\cal A}$ be an abelian category,  $X$ be an object in ${\cal A}$, and let $\cpx{Y}$ be a below-bounded complex over ${\cal A}$. Suppose that $m\in\mathbb{Z}$ and that $Y^i\in X^{\perp}$ for all $i<m$. Then $X[i]\in\mathscr{U}^q_{\cpx{Y}}$ for all $i\geq -m$.
\label{lemma-XcpxY}
\end{Lem}
 \begin{proof}
 For $i\geq m$, we have $-m\geq -i$, and $X[-m]\in\mathscr{U}^q_{Y^i[-i]}$.   For each $i<m$, since $Y^i\in X^{\perp}$, we have $X[-m]\in\mathscr{U}^q_{Y^i[-i]}$.   It follows that $X[-m]\in \mathscr{U}^q_{Y^i[-i]}$ for all $i\in\mathbb{Z}$.  Note that there is some integer $n<m$ such that $Y^i=0$ for all $i<n$, since $\cpx{Y}$ is bounded below.  Then $\subcpx{\leq m+1}{\cpx{Y}}$ is in $\{Y^{m+1}[-m-1]\}*\cdots *\{Y^n[-n]\}$.  By Lemma \ref{lemma-x1x2+m1m2} (2), we get that $X[-m]\in \mathscr{U}_{\subcpx{\leq m+1}{\cpx{Y}}}^q$.  Now it is clear that
$$\Hom_{\K{{\cal A}}}\big(X[-m], (\subcpx{>m+1}{\cpx{Y}})[i]\big)=0=\Hom_{\D{{\cal A}}}\big(X[-m], (\subcpx{>m+1}{\cpx{Y}})[i]\big)$$
for all $i\leq 1$. Hence $q_{(X[-m], (\subcpx{>m+1}{\cpx{Y}})[i])}$ is an isomorphism for all $i\leq 1$.  This establishes $X[-m]\in\mathscr{U}^q_{\subcpx{>m+1}{\cpx{Y}}}$. Since $\cpx{Y}$ is in $\{\subcpx{>m+1}{\cpx{Y}}\}*\{\subcpx{\leq m+1}{\cpx{Y}}\}$, we deduce that $X[-m]\in\mathscr{U}^q_{\cpx{Y}}$ by Lemma \ref{lemma-x1x2+m1m2} (2).  Finally, by definition,  we have $\mathscr{U}_{\cpx{Y}}^q[1]\subseteq \mathscr{U}^q_{\cpx{Y}}$. Hence $X[i]\in\mathscr{U}^q_{\cpx{Y}}$ for all $i\geq -m$.
\end{proof}

With the above lemmas, we can give a proof of Proposition \ref{prop-cpxXcpxY}.

\begin{proof}[Proof of Proposition \ref{prop-cpxXcpxY}]  What we need to prove is exactly $\cpx{X}\in\mathscr{U}^q_{\cpx{Y}}$. By Lemma \ref{lemma-XcpxY}, we have $X^i[-i]\in\mathscr{U}^q_{\cpx{Y}}$ for all $i\in\mathbb{Z}$.  Note that there is an integer $n$ such that $X^i=0$ for all $i>n$, since $\cpx{X}$ is above-bounded. Thus for each integer $m<n$, the complex $\subcpx{\geq m}{\cpx{X}}$ belongs to $\{X^n[-n]\}*\cdots*\{X^m[-m]\}$, and is consequently in $\mathscr{U}^q_{\cpx{Y}}$ by Lemma \ref{lemma-x1x2+m1m2} (1).  Taking $m$ to be  sufficiently small such that $Y^j=0$ for all $j<m+1$. Then for each integer $i\leq 1$, both $\Hom_{\K{{\cal A}}}(\subcpx{<m}{\cpx{X}}, \cpx{Y}[i])$ and $\Hom_{\D{{\cal A}}}(\subcpx{<m}{\cpx{X}}, \cpx{Y}[i])$ vanish.  Hence $q_{(\subcpx{<m}{\cpx{X}}, \cpx{Y}[i])}$ is an isomorphism for all $i\leq 1$, and consequently $\subcpx{<m}{\cpx{X}}\in\mathscr{U}^q_{\cpx{Y}}$.  Note that $\cpx{X}\in\{\subcpx{\geq m}{\cpx{X}}\}*\{\subcpx{<m}{\cpx{X}}\}$. It follows, by Lemma \ref{lemma-x1x2+m1m2} (1) again, that  $\cpx{X}\in\mathscr{U}^q_{\cpx{Y}}$.
\end{proof}

Proposition \ref{prop-cpxXcpxY} has the following useful corollary.

\begin{Koro}
Let $\mathcal{A}$ be an abelian category, and let $f: X\ra Y$ be a homomorphism in $\mathcal{A}$.  Suppose that $\cpx{Z}$ is a bounded complex over $\mathcal{A}$ such that $Z^i\in X^{\perp}$ for all $i<0$ and that $Z^i\in {}^{\perp}Y$ for all $i>0$. If $f$  factorizes through $\cpx{Z}$ in $\Db{\mathcal{A}}$, then $f$ factorizes through $Z^0$ in $\mathcal{A}$.
\label{corollary-factorthroughZ0}
\end{Koro}
\begin{proof}
Suppose that $f=gh$ for $g\in\Hom_{\Db{\mathcal{A}}}(X, \cpx{Z})$ and $h\in\Hom_{\Db{\mathcal{A}}}(\cpx{Z}, Y)$. By Proposition \ref{prop-cpxXcpxY}, both $g$ and $h$ can be presented by a chain map. Namely, $g=\cpx{g}$ and $h=\cpx{h}$ in $\Db{\mathcal{A}}$ for some chain maps $\cpx{g}: X\ra\cpx{Z}$ and $\cpx{h}:\cpx{Z}\ra Y$. Hence $f=\cpx{g}\cpx{h}=g^0h^0$ in $\Db{\mathcal{A}}$, and consequently $f=g^0h^0$ since $\mathcal{A}\hookrightarrow\Db{\mathcal{A}}$ is a fully faithful embedding.
\end{proof}


\section{The stable functor of a non-negative functor}\label{section-stablefunctor}

 The stable functor of a derived equivalence between Artin algebras was introduced in \cite{Hu2010}. In this section,  we greatly generalize this notion. Namely,
 we consider ``{\em non-negative functors}" between derived categories of abelian categories with enough projective objects, and develop a theory of their stable functors.

\medskip
Throughout this section, we assume that $\mathcal{A}$ and $\mathcal{B}$ are abelian categories with enough projective objects.  The full subcategories of projective objects are denoted by $\proja$ and $\projb$, respectively.  The corresponding stable categories are denoted by $\sta$ and $\stb$, respectively.

\subsection{Non-negative functors}

\begin{Def}
A triangle functor $F: \Db{\mathcal{A}}\lra \Db{\mathcal{B}}$  is called  {\bf uniformly bounded} if there are integers $r<s$ such that $F(X)\in\D[{[r, s]}]{\mathcal{B}}$ for all $X\in\mathcal{A}$, and is called {\bf non-negative} if $F$ satisfies the following conditions:

\smallskip
(1) $F(X)$ is isomorphic to a complex with zero homology in all negative degrees for all $X\in\mathcal{A}$.

(2) $F(P)$ is isomorphic to a complex in $\Kb{\projb}$ with zero terms in all negative degrees for all $P\in\proja$.
\label{def-uni-bounded-non-negative}
\end{Def}

{\parindent=0pt\it Remark.}
The condition (1) is equivalent to saying that $F$ sends objects in the part $\D[\geq 0]{\mathcal{A}}$ of the canonical $t$-structure  $(\D[\leq 0]{\mathcal{A}}, \D[\geq 0]{\mathcal{A}})$ of $\Db{\mathcal{A}}$ to objects in the part $\D[\geq 0]{\mathcal{B}}$ of the canonical $t$-structure  $(\D[\leq 0]{\mathcal{B}}, \D[\geq 0]{\mathcal{B}})$ of $\Db{\mathcal{B}}$. The condition (2) indicates that $F$ sends complexes in $\Kb{\proja}$ to complexes in $\Kb{\projb}$.

\medskip
For derived equivalences between module categories of rings, we have the following lemma. 
\begin{Lem}
Let $F: \Db{\Modcat{A}}\lra\Db{\Modcat{B}}$ be a derived equivalence between two rings $A$ and $B$. Then

$(1)$  $F$ is uniformly bounded.

$(2)$ $F$ is non-negative  if and only if the tilting complex associated to $F$ is isomorphic in $\Kb{\pmodcat{B}}$ to a complex with zero terms in all positive degrees. In particular, $F[i]$ is non-negative for sufficiently small $i$.
\label{lemma-derequiv-non-negative}
\end{Lem}
\begin{proof}
Let $\cpx{T}$ be a tilting complex associated to $F$, that is, $F(\cpx{T})\simeq B$. Since $\cpx{T}$ is a bounded complex, there are integers $r<s$ such that $T^i=0$ for all $i<r$ and for all $i>s$. Let $X$ be an $A$-module. There is an isomorphism
$$H^i(F(X))=\Hom_{\Db{\Modcat{B}}}(B, F(X)[i])\simeq\Hom_{\Db{\Modcat{A}}}(\cpx{T}, X[i])$$
for each integer $i$. It follows that $H^i(F(X))=0$ for all $i<r$ and for all $i>s$, that is, $F(X)\in\D[{[r, s]}]{\Modcat{B}}$.  This proves that $F$ is uniformly bounded.

By \cite[Proposition 6.2]{Rickard1989a}, the derived equivalence $F$ induces a triangle equivalence functor between $\Kb{\pModcat{A}}$ and $\Kb{\pModcat{B}}$.  Suppose that the tilting complex $\cpx{T}$ associated to $F$ has $T^i=0$ for all $i>0$.
By Lemma \ref{lemma-tiltCompForm}, the image $F(A)$ is isomorphic to a complex $\cpx{\bar{T}}\in\K[{[0, n]}]{\pmodcat{B}}$ for some non-negative integer $n$.  As an equivalence, the functor $F$ preserves coproducts. Hence $F(\coprod A)\in\K[{[0, n]}]{\pModcat{B}}$, and consequently $F(\pModcat{A})\subseteq \K[{[0, n]}]{\pModcat{B}}$. Finally, for each $A$-module $X$, we have
$\Hom_{\Db{\Modcat{B}}}(B, F(X)[i])\simeq \Hom_{\Db{\Modcat{A}}}(\cpx{T}, X[i])=0$
for all $i<0$. This implies that $H^i(F(X))=0$ for all $i<0$ and thus $F(X)\in\D[\geq 0]{\Modcat{B}}$. Hence $F$ is a non-negative functor.

Conversely, suppose that $F$ is a non-negative derived equivalence.  Then $F(A)$ is isomorphic to a bounded complex $\cpx{Q}$ in $\K[\geq 0]{\pModcat{B}}$. Let $\cpx{T}$ be a tilting complex associated to $F$, that is, $F(\cpx{T})\simeq B$. Then
$$\Hom_{\Db{\Modcat{A}}}(A, \cpx{T}[i])\simeq \Hom_{\Db{\Modcat{B}}}(F(A), B[i])=0$$
for all positive $i$. Hence $\cpx{T}$ has zero homology in all positive degrees. This shows that $\cpx{T}$ is split in all positive degrees and thus isomorphic to a complex in $\Kb{\pmodcat{A}}$ with zero terms in all positive degrees.
\end{proof}

In general, both statements in Lemma \ref{lemma-derequiv-non-negative} may fail for a triangle functor $F: \Db{\mathcal{A}}\ra \Db{\mathcal{B}}$ between the derived categories of abelian categories $\mathcal{A}$ and $\mathcal{B}$, even if $F$ is a derived equivalence.  
 For instance,  let $\mathcal{A}$ and $\mathcal{B}$ be the categories of finitely generated graded modules over the polynomial algebra $k[x_0, x_1,\cdots, x_n]$ and the exterior algebra $\bigwedge_k(e_0, e_1,\cdots, e_n)$, respectively. Then there is a triangle equivalence $F: \Db{\mathcal{A}}\lra\Db{\mathcal{B}}$, known as Koszul duality, such that $F(X\langle i\rangle)\simeq F(X)\langle -i \rangle [i]$ for all $X\in\Db{\mathcal{A}}$ and for all $i\in\mathbb{Z}$, where $\langle i\rangle$ is the degree shifting functor of graded modules.  The functor $F$ is not uniformly bounded and $F[i]$ cannot be non-negative for any $i\in\mathbb{Z}$.  Also the two notions in Definition  \ref{def-uni-bounded-non-negative} are independent.  Clearly, a uniformly bounded triangle functor $F$ needs not to be non-negative. The following example gives a non-negative functor which is not uniformly bounded.

{\bf\parindent=0pt Example}. Let $k$ be a field, and let $Q$ be the infinite quiver
$$\xymatrix{
\bullet &\ar[l]_(1){0}^{\alpha_1}_(0){1}\bullet &\bullet\ar[l]_(0){2}^{\alpha_2} &\bullet\ar[l]_(0){3}^{\alpha_3}&\ar[l]\cdots
}$$
A representation of $Q$ over $k$ is a collection of vector spaces $V_i$ for each vertex $i$ together with linear maps $f_{\alpha_i}: V_i\ra V_{i-1}$ for all $i$. Let $\mathcal{A}$ be the category of all finite dimensional representations $(V_i, f_{\alpha_i+1})_{i\geq 0}$ of $Q$
satisfying $f_{\alpha_i}f_{\alpha_{i-1}}=0$ for all $i>0$.
Let $P_0$ be the representation $k\lla 0\lla 0\lla\cdots$, and, for each $i>0$,  let $P_i$ be the representation $0\lla\cdots\lla k\llaf{1}k\lla 0\lla\cdots$, where the two $k$'s correspond to the vertices $i-1, i$. Then $\mathcal{A}$ is an abelian category with enough projective objects and  $P_i, i\geq 0$ are precisely those indecomposable projective objects in $\mathcal{A}$. Consider the following complexes over $\mathcal{A}$:
$$
\cpx{T}_i: \quad 0\lra P_0\lra \cdots\lra P_{i-1}\lra P_i\lra 0, \quad i\geq 0.$$
It is easy to check that $\{\cpx{T}_i|i\geq 0\}$ is a tilting subcategory of $\Db{\mathcal{A}}$, that is, the following two conditions are satisfied.

a) $\Hom_{\Db{\mathcal{A}}}(\cpx{T}_i, \cpx{T}_j[l])=0$  for all $i, j\in\mathbb{N}$ and $ l\neq 0$;

b) $\thick\{\cpx{T}_i|i\geq 0\}=\Db{\mathcal{A}}$.

{\parindent=0pt The } tilting subcategory $\{\cpx{T}_i|i\geq 0\}$  is equivalent as a category to the quiver $Q_T$:
$$\xymatrix{
\bullet\ar[r]^(0){0}^(1){1}_{\beta_1} &\bullet\ar[r]^(1){2}_{\beta_2} &\bullet\ar[r]^(1){3}_{\beta_3} &\bullet\ar[r] & \cdots
}$$
For each $i\geq 0$, let $P^*_i$ be the representation
$0\lra\cdots\lra 0\lra k\lraf{1}k\lraf{1}k\lra\cdots$, where the first $k$ corresponds to the vertex $i$.  Let $\mathcal{B}$ be the category of finitely generated representations of $Q_T$ over $k$. Then $\mathcal{B}$ is an abelian category with enough projective objects, and the indecomposable projective objects are $P_i^*, i\in\mathbb{N}$.  Note that $\gldim\mathcal{B}=1$ and $\Db{\mathcal{B}}=\Kb{\projb}$.  By \cite[Theorem 3.6]{Keller2006},  there is a triangle equivalence $F: \Db{\mathcal{B}}\lra\Db{\mathcal{A}}$ sending $P_i^*$ to $\cpx{T}_i$ for all $i\in\mathbb{N}$. This functor is non-negative, but not uniformly bounded.

\begin{Lem}
Let $\mathcal{A}$ and $\mathcal{B}$ be abelian categories with enough projective objects, and let $F:\Db{\mathcal{A}}\lra\Db{\mathcal{B}}$ be a uniformly bounded, non-negative triangle functor.  Suppose that $n>0$ is such that $F(\mathcal{A})\subseteq \D[{[0, n]}]{\mathcal{B}}$. Then

$(1)$ If $F$ admits a right adjoint $G$, then $G$ is uniformly bounded and $G(\mathcal{B})\subseteq\D[{[-n, 0]}]{\mathcal{A}}$.

$(2)$ If $F$ admits a left adjoint $E$, then $E(\projb)\subseteq\K[{[-n, 0]}]{\proja}$.

$(3)$ If $G$ is both a left adjoint and a right adjoint of $F$, then $G[-n]$ is  uniformly bounded and non-negative.
\label{lemma-unibounded-non-negative}
\end{Lem}
\begin{proof}
(1) Let $X$ be an object in $\mathcal{B}$ and $P$ be a projective object in $\mathcal{A}$. Then $\Hom_{\Db{\mathcal{A}}}(P, G(X)[i])\simeq\Hom_{\Db{\mathcal{B}}}(F(P), X[i])$ vanishes for all $i\not\in [-n, 0]$, since our assumption indicates that $F(P)$ is isomorphic to a complex in $\K[{[0, n]}]{\projb}$. It follows that $G(X)\in\D[{[-n, 0]}]{\mathcal{A}}$ for all $X\in \mathcal{B}$.

(2) Let $Q\in\projb$ and let $X$ be an object in $\mathcal{A}$. Then $\Hom_{\Db{\mathcal{A}}}(E(Q), X[i])\simeq\Hom_{\Db{\mathcal{B}}}(P, F(X)[i])$ vanishes for all $i\not\in [0, n]$.  This implies that $E(Q)\in\K[{[-n, 0]}]{\proja}$.

(3) This follows from (1) and (2) immediately.
 \end{proof}

 For the rest of this section, we assume that
$$F: \Db{\mathcal{A}}\lra\Db{\mathcal{B}}$$
is a non-negative triangle functor. The following lemma describes the images of objects in $\mathcal{A}$ under $F$.

\begin{Lem}
 For each $X\in \mathcal{A}$, there is a triangle
 $$\cpx{U}_X\lraf{i_X}F(X)\lraf{\pi_X} M_X\lraf{\mu_X}\cpx{U}_X[1]$$
 in $\Db{\mathcal{B}}$ with $M_X\in\mathcal{B}$ and $\cpx{U}_X\in\D[{[1,n_X]}]{\projb}$ for some $n_X>0$.
 \label{lemma-triangle-for-X}
 \end{Lem}
\begin{proof}
By definition, $F(X)$ has no homology in negative degrees. Take a projective resolution of $F(X)$ and then do good truncation at degree zero. The lemma follows.
\end{proof}

  \begin{Lem}
 Suppose that $\cpx{U}_i\lraf{\alpha_i}\cpx{X}_i\lraf{\beta_i}M_i\lraf{\gamma_i}\cpx{U}_i[1], i=1, 2$ are triangles in $\Db{\mathcal{B}}$ such that $M_1, M_2$ are objects in $\mathcal{B}$ and $\cpx{U}_1, \cpx{U}_2\in\D[{[1, n]}]{\projb}$.  Then, for each morphism $f: \cpx{X}_1\lra\cpx{X}_2$ in $\Db{\mathcal{B}}$, there is morphism $b: M_1\lra M_2$ in $\mathcal{B}$ and a morphism $a: \cpx{U}_1\lra\cpx{U}_2$  in $\Db{\mathcal{B}}$ such that the diagram
  $$\xymatrix@M=1.5mm{
 \cpx{U}_1\ar[r]^{\alpha_1}\ar[d]^{a} & \cpx{X}_1\ar[r]^{\beta_1}\ar[d]^{f} & M_1\ar[r]^{\gamma_1}\ar[d]_{b} &\cpx{U}_1[1]\ar[d]^{a[1]}\\
 \cpx{U}_2\ar[r]^{\alpha_2} & \cpx{X}_2\ar[r]^{\beta_2} & M_2\ar[r]^{\gamma_2} &\cpx{U}_2[1]\\
 }$$
 is commutative. Moreover, if $f$ is an isomorphism in $\Db{\mathcal{B}}$, then $\underline{b}$ is an isomorphism in $\stb$.
 \label{lemma-exist-ab}
 \end{Lem}
 \begin{proof}
 The morphisms $a$ and $b$ exist because $\alpha_1f\beta_2$ must be zero, since
 $$\Hom_{\Db{\mathcal{B}}}(\cpx{U}_1, M_2)\simeq\Hom_{\Kb{\mathcal{B}}}(\cpx{U}_1, M_2)=0.$$
Now assume that $f$ is an isomorphism in $\Db{\mathcal{B}}$. Namely, there is a morphism $g: \cpx{X}_2\lra\cpx{X}_1$ in $\Db{\mathcal{B}}$ such that $fg=1_{\cpx{X}_1}$ and $gf=1_{\cpx{X}_2}$. By the above discussion, there a morphism $c: M_2\lra M_1$ such that $\beta_2c=g\beta_1$.  Then
$$\beta_1-\beta_1bc=\beta_1-f\beta_2c=\beta_1-fg\beta_1=0, $$
and $1_{M_1}-bc$ factorizes through $\cpx{U}_1[1]$. It follows that $1_{M_1}-bc$ factorizes through the projective object $U_1^1$ by Corollary \ref{corollary-factorthroughZ0}. Hence $\underline{b}\underline{c}=\underline{1_{M_1}}$ is the identity map of $M_1$ in $\stb$.  Similarly we have $\underline{c}\underline{b}=\underline{1_{M_2}}$, and therefore $\underline{b}: M_1\lra M_2$ is an isomorphism in $\stb$.
 \end{proof}

 \subsection{The definition of the stable functor}

Keeping the notations above, we can define a functor $\bar{F}: \sta\lra \stb$
  as follows. For each $X\in\mathcal{A}$, we fix a triangle  $$\xi_{X}: \quad \cpx{U}_X\lraf{i_X}F(X)\lraf{\pi_X} M_X\lraf{\mu_X}\cpx{U}_X[1]$$
in $\Db{\mathcal{B}}$ with $M_X\in\mathcal{B}$, and $\cpx{U}_X$ a complex in $\D[{[1, n_X]}]{\projb}$ for some $n_X>0$.  The existence is guaranteed by Lemma \ref{lemma-triangle-for-X}. For each morphism $f: X\ra Y$ in $\mathcal{A}$,  by Lemma \ref{lemma-exist-ab}, we can form a commutative diagram in $\Db{\mathcal{B}}$:
 $$\xymatrix@M=1.5mm{
 \cpx{U}_X\ar[r]^{i_X}\ar[d]^{a_f} & F(X)\ar[r]^{\pi_X}\ar[d]^{F(f)} & M_X\ar[r]^{\mu_X}\ar[d]_{b_f} &\cpx{U}_X[1]\ar[d]^{a_f[1]}\\
 \cpx{U_Y}\ar[r]^{i_Y} & F(Y)\ar[r]^{\pi_Y} & M_Y\ar[r]^{\mu_Y} &\cpx{U_Y}[1]\\
 }$$
If $b'_f$ is another morphism such that $\pi_Xb'_f=F(f)\pi_Y$, then $\pi_X(b_f-b'_f)=0$, and $b_f-b'_f$ factorizes through $\cpx{U}_X[1]$. By Corollary \ref{corollary-factorthroughZ0}, the map $b_f-b'_f$ factorizes through $U_X^1$ which is projective. Hence the morphism $\underline{b_f}\in\stb(M_X, M_Y)$ is uniquely determined by $f$. Moreover, suppose that $f$ factorizes through a projective object $P$ in $\mathcal{A}$, say $f=gh$ for $g: X\ra P$ and $h: P\ra Y$.  Then $\pi_X(b_f-b_gb_h)=F(f)\pi_Y-F(g)\pi_Pb_h=F(f)\pi_Y-F(g)F(h)\pi_Y=0$.  Hence $b_f-b_gb_h$ factorizes through $\cpx{U}_X[1]$, and factorizes through $U_X^1$ by Corollary \ref{corollary-factorthroughZ0}.  Thus $b_f$ factorizes through $P\oplus U_X^1$ which is projective.  Hence $\underline{b_f}=0$.  Thus, we get a well-defined map
$$\phi: \Hom_{\sta}(X, Y)\lra \Hom_{\stb}(M_X, M_Y), \quad \underline{f}\mapsto \underline{b_f}. $$
It is easy to say that $\phi$ is functorial in $X$ and $Y$.  Defining $\bar{F}(X):=M_X$ for each $X\in\mathcal{A}$ and $\bar{F}(\underline{f}):=\phi(\underline{f})$ for each morphism $f$ in $\mathcal{A}$, we get a functor
$$\bar{F}: \sta\lra\stb $$
which is called {\em the stable functor} of $F$.

\medskip
{\parindent=0pt\bf  Example.} (a). If $k$ is a field, and if $F=\cpx{\Delta}\otimesL_A-$ is a standard derived equivalence given by a two-sided tilting complex $\cpx{\Delta}$ of $B$-$A$-bimodules. Assume that $\cpx{\Delta}$ has no homology in negative degrees. Take a projective resolution of $\cpx{\Delta}$ and do good truncation at degree zero.  Then $\cpx{\Delta}$ is isomorphic in $\Db{B\otimes_kA\opp}$ to a complex of the form
$$0\lra M\lra P^1\lra\cdots\lra P^n\lra 0$$
with $P^i$ projective for all $i>0$. By \cite[Proposition 3.1]{Rickard1991}, this complex is a one-sided tilting complex on both sides. It follows that ${}_BM_A$ is projective as one-sided modules, and $F(X)$ is isomorphic to $0\lra M\otimes_AX\lra P^1\otimes_AX\lra \cdots\lra P^n\otimes_AX\lra 0$ with $P^i\otimes_AX$ projective for all $i>0$. In this case, the stable functor $\bar{F}$ of $F$ is induced by the exact functor ${}_BM\otimes_A-: \Modcat{A}\lra\Modcat{B}$.

\medskip
 (b). Let $\mathcal{A}$ be an abelian category with enough projective objects, and let $n$ be a non-negative integer. The $n$th syzygy functor $\Omega_{\mathcal{A}}^n: \sta\lra\sta$ is 
a stable functor of the derived equivalence $[-n]: \Db{\mathcal{A}}\lra\Db{\mathcal{A}}$.

\begin{Prop}
The following diagram is commutative up to isomorphism.
$$\xymatrix@M=2mm{
\sta\ar[d]^{\bar{F}} \ar[r]^(.3){\Sigma} & \Db{\mathcal{A}}/\Kb{\proja}\ar[d]^{F}\\
\stb \ar[r]^(.3){\Sigma} & \Db{\mathcal{B}}/\Kb{\projb},
}$$
where $\Sigma$ is the canonical functor induced by the embedding $\mathcal{A}\hookrightarrow \Db{\mathcal{A}}$.
\label{proposition-stFun-commdiag}
\end{Prop}
\begin{proof}
For each $X\in\mathcal{A}$, the morphism $\pi_X: F(X)\ra \bar{F}(X)$ in $\Db{\mathcal{B}}$ can be viewed as a morphism $\pi_X: F\Sigma(X)\lra \Sigma \bar{F}(X)$ in $\Db{\mathcal{B}}/\Kb{\projb}$. We claim that this gives a natural isomorphism from $F\circ\Sigma$ to $\Sigma\circ\bar{F}$.  Since $\cpx{U}_X$ in the triangle $\xi_X$ is a complex in $\Kb{\projb}$, the morphism $\pi_X$ is an isomorphism in $\Db{\mathcal{B}}/\Kb{\projb}$.  Moreover, for each morphism $f: X\ra Y$ in $\mathcal{A}$, one can check from the definition of $\bar{F}$ that there is a commutative diagram
$$\xymatrix{
F\circ \Sigma(X)\ar[r]^{\pi_X}\ar[d]^{F\circ\Sigma (\underline{f})} & \Sigma\circ \bar{F}(X)\ar[d]^{\Sigma\circ\bar{F}(\underline{f})}\\
F\circ \Sigma(Y)\ar[r]^{\pi_X} & \Sigma\circ \bar{F}(Y).\\
}$$
This finishes the proof.
\end{proof}

\subsection{Uniqueness of the stable functor}
From the definition of the stable functor, it is unclear that whether the stable functor is independent of the choices of the triangles $\xi_X$. In this subsection, we shall solve this problem. Actually, we will show that  isomorphic  non-negative functors have isomorphic stable functors.

\smallskip
We keep the notations in the previous subsection.  For each object $X\in\mathcal{A}$, suppose that we choose and fix  another triangle
$$\xi'_{X}:\quad  \cpx{U'_{X}}\lraf{i'_{X}}F(X)\raf{\pi'_{X}} M'_{X}\lraf{\mu'_{X}} \cpx{U'_{X}}[1]$$
in $\Db{\mathcal{B}}$ with $M'_{X}\in \mathcal{B}$ and  $\cpx{U'_{X}}$ a complex in $\D[{[1, n'_X]}]{\projb}$ for some $n'_X>0$. Let $\bar{F}': \sta\lra\stb$ be the functor defined by using the triangles $\xi'_{X}$'s.  That is, $\bar{F}'(X)=M'_{X}$ for each $X\in\mathcal{A}$, and $\bar{F}'(\underline{f})=\underline{b'_{f}}$ for momorphism $f: X\ra Y$ in $\mathcal{A}$, where $b'_{f}: M'_{X}\ra M'_{Y}$ is a morphism in $\mathcal{B}$ such that $\pi'_{X}b'_{f}=F(f)\pi'_{Y}$.

 \begin{Prop}
 The functors $\bar{F}$ and $\bar{F}'$ are isomorphic.
\label{proposition-unique-choices}
\end{Prop}
\begin{proof}
   For each $X\in\mathcal{A}$, by Lemma \ref{lemma-exist-ab}, we can form a commutative diagram
   $$\xymatrix{
 \cpx{U}_X\ar[r]^{i_X}\ar[d]^{\alpha_X} & F(X)\ar[r]^{\pi_X}\ar@{=}[d] & M_X\ar[r]^{\mu_X}\ar[d]_{\eta_X} &\cpx{U}_X[1]\ar[d]^{\alpha_X[1]}\\
 \cpx{U'}_{X}\ar[r]^{i'_X} & F(X)\ar[r]^{\pi'_X} & M'_X\ar[r]^{\mu'_X} &\cpx{U'}_X[1], \\
 }$$
in $\Db{\mathcal{B}}$ such that $\underline{\eta_X}$ is an isomorphism in $\stb$.  Now, for each  morphism $f: X\ra Y$ in $\mathcal{A}$, we have
 $$\begin{array}{rl}
     \pi_X b_f\eta_Y & = F(f)\pi_Y\eta_Y\\
                   & =  F(f) \pi'_Y\\
                   &= \pi'_X b'_f\\
                   &=\pi_X\eta_Xb'_f.
 \end{array}$$
 Hence $\pi_X(\eta_Xb'_f- b_f\eta_Y)=0$, and $\eta_Xb'_f- b_f\eta_Y$ factorizes through $\cpx{U}_X[1]$. It follows that $\eta_Xb'_f- b_f\eta_Y$ factorizes through the projective object $U_X^1$  by  Corollarly \ref{corollary-factorthroughZ0}.  This shows that  $\underline{\eta_X}\underline{b'_f}-\underline{b_f}\underline{\eta_Y}=0$, that is,  $$\underline{\eta_X}\bar{F}'(\underline{f})=\bar{F}(\underline{f})\underline{\eta_Y}.$$
Thus, we get a natural transformation $\underline{\eta}: \bar{F}\ra\bar{F}'$ with $\underline{\eta}_X:=\underline{\eta_X}$ for all $X\in\mathcal{A}$. Since we have shown that $\underline{\eta_X}$ is an isomorphism for all $X\in\mathcal{A}$, it follows that $\underline{\eta}: \bar{F}\lra \bar{F}'$ is an isomorphism of functors.
\end{proof}

 The above proposition shows that, up to isomorphism,  the stable functor $\bar{F}$  is independent of the choices of the triangles $\xi_X$'s, and is uniquely determined by $F$. Actually, we can further prove the following theorem.
\begin{Theo}
  Let $F_1, F_2: \Db{\mathcal{A}}\lra\Db{\mathcal{B}}$ be two isomorphic non-negative triangle functors. Then their stable functors $\bar{F}_1$ and $\bar{F}_2$ are isomorphic.
 \label{theorem-unique-F1F2}
 \end{Theo}
 \begin{proof}
 Suppose that $\eta: F_1\ra F_2$ is an isomorphism of triangle functors.  For each $X\in\mathcal{A}$,  by definition, we have two triangles
 $\cpx{U}_X\lraf{i_{X}}F_1(X)\lraf{\pi_X}\bar{F}_1(X)\lraf{\mu_X}\cpx{U}_X[1]$ and $\cpx{V}_X\lraf{j_X}F_2(X)\lraf{p_X}\bar{F}_2(X)\lraf{\rho_X}\cpx{V}_X[1]$ with  $\cpx{U}_X$ and $\cpx{V}_X$ in $\D[{[1, n_X]}]{\projb}$ for some positive integer $n_X$.  By Lemma \ref{lemma-exist-ab},  we can form a commutative diagram
   $$\xymatrix{
 \cpx{U}_X\ar[r]^{i_X}\ar[d]^{\alpha_X} & F_1(X)\ar[r]^{\pi_X}\ar[d]_{\eta_X} & \bar{F}_1(X)\ar[r]^{\mu_X}\ar[d]_{\delta_X} &\cpx{U}_X[1]\ar[d]^{\alpha_X[1]}\\
 \cpx{V}_{X}\ar[r]^{j_X} & F_2(X)\ar[r]^{p_X} & \bar{F}_2(X)\ar[r]^{\mu'_X} &\cpx{V}_X[1] \\
 }$$
 such that  $\underline{\delta_X}$ is an isomorphism in $\stb$, since $\eta_X$ is an  isomorphism in $\Db{\mathcal{B}}$.  Now for each morphism $f: X\lra Y$ in $\mathcal{A}$, there are  morphisms $b_f: \bar{F}_1(X)\lra\bar{F}_1(Y)$ and $b'_f: \bar{F}_2(X)\lra\bar{F}_2(Y)$ with $$\pi_Xb_f=F_1(f)\pi_Y, \quad p_Xb'_f=F_2(f)p_Y$$
 such that $\bar{F}_1(\underline{f})=\underline{b_f}$ and $\bar{F}_2(\underline{f})=\underline{b'_f}$.  Now we have
 $$\begin{array}{rl}
   \pi_X(b_f\delta_Y-\delta_Xb'_f) &= F_1(f)\pi_Y\delta_Y- \eta_Xp_Xb'_f\\
   & = F_1(f)\eta_Yp_Y-\eta_XF_2(f)p_Y\\
   &=\big( F_1(f)\eta_Y-\eta_XF_2(f)\big)p_Y=0.
 \end{array}$$
 Hence $b_f\delta_Y-\delta_Xb'_f$ factorizes through $\cpx{U}_X[1]$, and consequently factorizes through the projective object $U_X^1$ by Corollary \ref{corollary-factorthroughZ0}.  Hence $0=\underline{b_f}\underline{\delta_Y}-\underline{\delta_X}\underline{b'_f}=\bar{F}_1(\underline{f})\underline{\delta_Y}-\underline{\delta_X}\bar{F}_2(\underline{f})$.  This gives rise to a natural transformation $\underline{\delta}: \bar{F}_1\lra \bar{F}_2$ with $\underline{\delta}_X:=\underline{\delta_X}$ for each $X\in\mathcal{A}$.  Recall that $\underline{\delta_X}$ is an isomorphism for all $X\in\mathcal{A}$. This proves that $\underline{\delta}: \bar{F}_1\lra\bar{F}_2$ is an isomorphism.
\end{proof}

\subsection{The composition of stable functors}

Suppose that $\mathcal{A}, \mathcal{B}$ and $\mathcal{C}$ are abelian categories with enough projective objects.  Let $F: \Db{\mathcal{A}}\lra\Db{\mathcal{B}}$ and $G: \Db{\mathcal{B}}\lra\Db{\mathcal{C}}$ be non-negative triangle functors. It is easy to see that  $GF$ is also non-negative. The relationship among the stable functors of $F, G$ and $GF$ is the following theorem.
\begin{Theo}
The functors $\bar{G}\circ\bar{F}$ and $\overline{GF}$ are isomorphic.
\label{theorem-stFunctor-composite}
\end{Theo}
\begin{proof}
For each $X\in\mathcal{A}$, there are two triangles
$$\cpx{U}_X\lraf{i_X}F(X)\lraf{\pi_X}\bar{F}(X)\lraf{\mu_X}\cpx{U}_X[1], $$
$$\cpx{V}_{X}\lraf{j_X}G(\bar{F}(X))\lraf{p_X}\bar{G}\bar{F}(X)\lraf{\omega_X}\cpx{V}_X[1]$$
with $\cpx{U}_X\in\D[{[1, n_X]}]{\projb}$ and $\cpx{V}_X\in\D[{[1, m_X]}]{\projc}$.
By the octahedral axiom, we can form a commutative diagram
$$\xymatrix@M=1.5mm{
&&\cpx{V}_X\ar@{=}[r]\ar[d]^{j_X} &\cpx{V}_X\ar[d]^{\epsilon_X}\\
G(\cpx{U}_X)\ar[d]^{u_X} \ar[r]^{G(i_X)} & GF(X)\ar[r]^{G(\pi_X)}\ar@{=}[d] &G(\bar{F}(X))\ar[r]^{G(\mu_X)}\ar[d]^{p_X} &G(\cpx{U}_X)[1]\ar[d]^{u_X[1]}\\
\cpx{W}_X\ar[r]^{\alpha_X} &GF(X)\ar[r]^{\beta_X}& \bar{G}\bar{F}(X)\ar[r]^{\gamma_X}\ar[d]^{\omega_X} &\cpx{W}_X[1]\ar[d]^{v_X}\\
&& \cpx{V}_X[1]\ar@{=}[r]& \cpx{V}_X[1]
}$$
with all rows and columns being triangles in $\Db{\mathcal{C}}$. Since $\cpx{U}_X$ is a complex in $\D[{[1, n_X]}]{\projb}$, and $G(U_X^i)$ is isomorphic to a complex $\D[{[0, t_i]}]{\projc}$ for all $i=1, \cdots, n_X$, it follows, by \cite[Lemma 2.1]{Hu2012} for example, that $G(\cpx{U}_X)$ is isomorphic in $\Db{\mathcal{C}}$ to a complex in $\D[{[1,a_X]}]{\projc}$ for some $a_X>0$. Recall that $\cpx{V}_X$ is a complex in $\D[{[1, m_X]}]{\projc}$.  As a result, the complex $\cpx{W}_X[1]$, which is the mapping cone of $\epsilon_X$, is isomorphic to a complex in $\D[{[0, m_X+a_X-1]}]{\projc}$. Hence we can assume that $\cpx{W}_X$ is  a complex in $\D[{[1, m_X+a_X]}]{\projc}$.  Thus  the stable functor of $GF$ can be defined by fixing, for each  $X\in\mathcal{A}$,  the triangle
$$\cpx{W}_X\lraf{\alpha_X}GF(X)\lraf{\beta_X}\bar{G}\bar{F}(X)\lraf{\gamma_X}\cpx{W}_X[1].$$
Therefore, for each $X\in\mathcal{A}$, we have $\overline{GF}(X)=\bar{G}\bar{F}(X)$.

Let $f: X\lra Y$ be a morphism in $\mathcal{A}$. By the construction of stable functor, there is a morphism $b_f: \bar{F}(X)\lra\bar{F}(Y)$ in $\mathcal{B}$ such that $\pi_Xb_f=F(f)\pi_Y$, and $\bar{F}(\underline{f})=\underline{b_f}$, and  there is a  morphism
$$c_f: \bar{G}\bar{F}(X)\lra\bar{G}\bar{F}(Y)$$
in $\mathcal{C}$ such that $\beta_Xc_f=GF(f)\beta_Y$ and $\overline{GF}(\underline{f})=\underline{c_f}$.  Also there is a morphism
$$c'_f: \bar{G}\bar{F}(X)\lra\bar{G}\bar{F}(Y)$$
in $\mathcal{C}$ such that $p_Xc'_f=G(b_f)p_Y$ and $\bar{G}(\underline{b_f})=\underline{c'_f}$. Now we have the following
$$\begin{array}{rl}
\beta_X(c_f-c'_f) &= GF(f)\beta_Y-G(\pi_X)p_Xc'_f\\
                           &= GF(f)\beta_Y-G(\pi_X)G(b_f)p_Y\\
                           &=GF(f)\beta_Y-G(\pi_Xb_f)p_Y\\
                           &= GF(f)\beta_Y-G(F(f)\pi_Y)p_Y\\
                           &= GF(f)\beta_Y-GF(f)G(\pi_Y)p_Y\\
                           &=  GF(f)\beta_Y-GF(f)\beta_Y=0.
\end{array}$$
Hence $c_f-c'_f$ factorizes through $\cpx{W}_X[1]$, and consequently $c_f-c'_f$ factorizes through the projective object $W_X^1$ by Corollary \ref{corollary-factorthroughZ0}. Therefore we have $\underline{c_f}=\underline{c'_f}$, and $$\overline{GF}(\underline{f})=\underline{c_f}=\underline{c'_f}=\bar{G}(\underline{b_f})=\bar{G}\bar{F}(\underline{f}).$$
This shows that, by choosing the triangles carefully, we get $\bar{G}\circ\bar{F}=\overline{GF}$. Since the stable functor is unique up to isomorphism,  we are done.
\end{proof}

An immediate consequence is the following.
\begin{Koro}
Keep the notations above. The functors $\bar{F}\circ\Omega_{\mathcal{A}}\simeq \Omega_{\mathcal{B}}\circ \bar{F}$.
\label{corollary-comm-with-omega}
\end{Koro}
\begin{proof}
Since $F\circ [-1]\simeq [-1]\circ F$, the corollary follows from Theorem \ref{theorem-unique-F1F2} and Theorem \ref{theorem-stFunctor-composite}.
\end{proof}

\subsection{Exactness of the stable functor}

Although it is hard to say whether the stable functor $\bar{F}$ is an exact functor or not,  the following proposition shows that the stable functor does have certain ``exactness" property.

\begin{Prop}
Keep the notations above.  Suppose that $0\ra X\raf{f}Y\raf{g} Z\ra 0$ is an exact sequence in $\mathcal{A}$.  Then there is an exact sequence
$$0\lra \bar{F}(X)\lraf{[a, b]} \bar{F}(Y)\oplus P\lraf{\left[\begin{smallmatrix}u&v\\s&t \end{smallmatrix}\right]} \bar{F}(Z)\oplus Q\lra 0$$
in $\mathcal{B}$ for some projective objects $P$ and $Q$ such that $\bar{F}(\underline{f})=\underline{a}$ and $\bar{F}(\underline{g})=\underline{u}$.
\label{proposition-stfunctor-property-exseq}
\end{Prop}
\begin{proof}
For each $X\in\mathcal{A}$, since $F$ is a non-negative functor, we may assume that $F(X)$ is a complex $\cpx{P}_X\in\D[{[0, n_X]}]{\mathcal{B}}$ with $P_X^0=\bar{F}(X)$ and $P_X^i\in\projb$ for all $i>0$.

From the exact sequence $0\ra X\raf{f}Y\raf{g} Z\ra 0$ in $\mathcal{A}$, we get a triangle in $ \Db{\mathcal{A}}$:
$$X\raf{f}Y\raf{g} Z\raf{h} X[1].$$
Applying the functor $F$ results in a  triangle in $\Db{\mathcal{B}}$:
$$\cpx{P}_X\raf{F(f)}\cpx{P}_Y\raf{F(g)} \cpx{P}_Z\raf{F(h)} \cpx{P}_X[1].$$
By Proposition \ref{prop-cpxXcpxY}, the morphisms $F(f)$ and $F(g)$ are induced by chain maps $\cpx{p}$ and $\cpx{q}$, respectively. That is,  $F(f)=\cpx{p}$ and $F(g)=\cpx{q}$. There is a commutative diagram in $\Db{\mathcal{B}}$:
 $$\xymatrix@M=2mm{
 \cpx{P}_Z[-1]\ar[r]\ar@{=}[d] & \cpx{P}_X\ar[r]^{\cpx{p}}\ar[d]_{r} & \cpx{P}_Y\ar[r]^{\cpx{p}}\ar@{=}[d] &\cpx{P}_Z\ar@{=}[d]\\
 \cpx{P}_Z[-1]\ar[r] & \con(\cpx{q})[-1]\ar[r]^(.65){\cpx{\pi}} & \cpx{P}_Y\ar[r]^{\cpx{q}} &\cpx{P}_Z, \\
 }$$
for some isomorphism $r$, where $\cpx{\pi}=(\pi^i)$ with $\pi^i:P_Y^i\oplus P_Z^{i-1}\ra P_Y^i$ the canonical projection for each integer $i$.
By Proposition \ref{prop-cpxXcpxY}, the morphism $r$ is induced by a chain map $\cpx{r}$. Then  cone($\cpx{r}$) is of the form
$$\xymatrix@C=15mm@M=3mm{
0\ar[r] & P_X^0\ar[r]^(.42){[-d, r^0]} & P_X^1\oplus P_Y^0\ar[r]^(.45){\left[\begin{smallmatrix}-d&x&y\\0&-d&q^0 \end{smallmatrix}\right]} & P_X^2\oplus P_Y^1\oplus P_Z^0\ar[r] & \cdots,
}$$
where $r^1=[x, y]:P_X^1\ra P_Y^1\oplus P_Z^0$ and  $P_X^i, P_Y^i$ and $P_Z^i$ are projective for $i\geq 1$.
Since $r=\cpx{r}$ is an isomorphism in $\Db{\mathcal{A}}$, the mapping cone $\con(\cpx{r})$ is an acyclic complex. Thus dropping the split direct summands of $\con(\cpx{r})$, we get an exact sequence
$$\xymatrix@C=8mm@M=3mm{
0\ar[r] & P_X^0\ar[r]^(.42){[-d, r^0]} & P_X^1\oplus P_Y^0\ar[r]^{\left[\begin{smallmatrix}\alpha&\beta\\ \gamma&q^0 \end{smallmatrix}\right]} & Q \oplus P_Z^0\ar[r]^(.55){\left[\begin{smallmatrix}\delta\\ \eta\end{smallmatrix}\right]} & V\ar[r] &0,
}$$
where $Q=P_X^2\oplus P_Y^1$ and $V$ is a projective object in $\mathcal{B}$. Let $[\epsilon,\chi]: V\lra Q\oplus P_Z^0$ be such that $[\epsilon,\chi]\left[\begin{smallmatrix}\delta\\ \eta\end{smallmatrix}\right]=1_V$. We claim that the sequence
$$\xymatrix@C=10mm@M=3mm{
0\ar[r] & P_X^0\ar[r]^(.38){[0, -d, r^0]} & V\oplus P_X^1\oplus P_Y^0\ar[r]^(.55){\left[\begin{smallmatrix}\epsilon&\chi\\ \alpha& \beta\\ \gamma &q^0 \end{smallmatrix}\right]} & Q \oplus P_Z^0\ar[r] &0,
}$$
is exact.  It suffices to prove that the sequence is exact at the middle term.  Clearly $[0, -d, r^0]\left[\begin{smallmatrix}\epsilon& \chi\\\alpha&\beta\\ \gamma&q^0 \end{smallmatrix}\right]=0$. If $[x_1, x_2, x_3]$ is a morphism from an object $U$ to $V\oplus P_X^1\oplus P_Y^0$ such that $[x_1, x_2, x_3]\left[\begin{smallmatrix}\epsilon&\chi\\ \alpha& \beta\\ \gamma &q^0 \end{smallmatrix}\right]=0$, then $x_1=[x_1, x_2, x_3]\left[\begin{smallmatrix}\epsilon&\chi\\\alpha&\beta\\ \gamma&q^0 \end{smallmatrix}\right]\left[\begin{smallmatrix}\delta\\\eta\end{smallmatrix}\right]=0$, and consequently $[x_2, x_3]\left[\begin{smallmatrix}\alpha&\beta\\ \gamma&q^0 \end{smallmatrix}\right]=0$. Thus $[x_2, x_3]$ factorizes uniquely through  $[-d, r^0]$ by exactness, and $[x_1, x_2, x_3]=[0, x_2, x_3]$ factorizes through $[0, -d, r^0]$.
Setting $P=V\oplus P_X^1, a=r^0, b=[0, -d], u=q^0, v=\gamma, s=\left[\begin{smallmatrix}\chi \\ \beta\end{smallmatrix}\right]$ and $t=\left[\begin{smallmatrix}\epsilon\\\alpha\end{smallmatrix}\right]$, we get the desired exact seuqence.
\end{proof}


\section{Gorenstein projective objects under the stable functor}\label{section-Gorenstein}

Let $\mathcal{A}$ be an abelian category with enough projective objects, and let $\proja$ be the full subcategory of $\mathcal{A}$ consisting of all projective objects.  An object $X\in\mathcal{A}$ is called {\em Gorenstein projective} if there is an exact sequence
$\cpx{P}$:
$$\cdots\lra
P^{-1}\lraf{d^{-1}}P^{0}\lraf{d^{0}}
P^{1}\lraf{d^{1}}\cdots
$$
in $\C{\proja}$ such that $\HomP_{\mathcal{A}}(\cpx{P}, Q)$ is exact for all $Q\in\proja$ and $X\simeq \Img d^0$.
We denote by $\GProj{\mathcal{A}}$ the full subcategory consisting of all Gorenstein projective objects. Then $\GProj{\mathcal{A}}$ is a Frobenius category with projective (=injective) objects being the projective objects in $\mathcal{A}$. The stable category $\stGProj{\mathcal{A}}$ is a triangulated category with shifting functor $\Omega_{\mathcal{A}}^{-1}$.  The following lemma is an alternative description of Gorenstein projective objects.

\begin{Lem} Let $\mathcal{A}$ be an abelian category with enough projective objects. Then  an object  $X\in\mathcal{A}$ is Gorenstein projective if and only if there  are short exact sequences
 $$ 0\lra X^i\lra P^{i+1}\lra X^{i+1}\lra 0$$
in $\mathcal{A}$ with $P^i$ projective  and $X^i\in{}^{\perp}\proja$ for $i\in\mathbb{Z}$ such that  $X^0=X$.
 \label{lemma-gproj-ses}
\end{Lem}

The following proposition shows that the stable functor of certain non-negative functor preserves Gorenstein projective modules.
\begin{Prop}
Let $\mathcal{A}$ and $\mathcal{B}$ be two abelian categories with enough projective objects.   Suppose that $F: \Db{\mathcal{A}}\lra \Db{\mathcal{B}}$ is a non-negative triangle functor admitting a right adjoint $G$ with $G(Q)\in\Kb{\proja}$ for all $Q\in\projb$. Let $m$ be a non-negative integer. Then we have the following.

$(1)$. If $X\in {}\perpg{m}\proja$, then $\bar{F}(X)\in {}\perpg{m}\projb$.

$(2)$. If $X\in \GProj{\mathcal{A}}$, then $\bar{F}(X)\in \GProj{\mathcal{B}}$.
\label{proposition-gCM}
\end{Prop}
\begin{proof}
For each $Q\in\projb$, by assumption $G(Q)\in\Kb{\proja}$. We claim that $G(Q)$ is isomorphic to a complex in $\Kb{\proja}$ with zero terms in all positive degrees. This
is equivalent to saying that  $\Hom_{\Db{\mathcal{A}}}(P, G(Q)[i])=0$
for all $P\in\proja$ and all $i>0$. However, this follows from the isomorphism $$\Hom_{\Db{\mathcal{A}}}(P, G(Q)[i])\simeq \Hom_{\Db{\mathcal{B}}}(F(P), Q[i])$$ and the assumption that $F$ is non-negative.

(1). Suppose that $X\in {}\perpg{m}\proja$. Then $\proja\subseteq X\perpg{m}$.  It is clear that $X\perpg{m}$ is closed under the shift functor $[1]$ and extensions. It follows that  each bounded complex in $\Kb{\proja}$, which has zero terms in all positive degrees, are in $X\perpg{m}$.  In particular, $G(Q)\in X\perpg{m}$ for all $Q\in\projb$. By the definition of the stable functor $\bar{F}$, there is a triangle
 $$\cpx{U}_X\lraf{i_X}F(X)\lraf{\pi_X} \bar{F}(X)\lraf{\mu_X}\cpx{U}_X[1]$$
in $\Db{\mathcal{B}}$ with $\cpx{U}_X\in\D[{[1,n]}]{\projb}$ for some $n>0$.  Let $Q\in\projb$, and let $i$ be a positive integer. We have $$\Hom_{\Db{\mathcal{B}}}(\cpx{U}_X[1], Q[i])=0=\Hom_{\Db{\mathcal{B}}}(\cpx{U}_X, Q[i]).$$
Applying $\Hom_{\Db{\mathcal{B}}}(-, Q[i])$ to the above triangle results in an isomorphism
$$\Hom_{\Db{\mathcal{B}}}(\bar{F}(X), Q[i])\simeq \Hom_{\Db{\mathcal{B}}}(F(X), Q[i]).$$
The latter is further isomorphic to $\Hom_{\Db{\mathcal{A}}}(X, G(Q)[i])$, which vanishes for $i>m$. Hence $\bar{F}(X)\in {}\perpg{m} \projb$.

(2).  Suppose that $X$ is Gorenstein projective. By Lemma \ref{lemma-gproj-ses}, there are short exact sequences
$$0\lra X^{i-1}\lra P^{i}\lra X^{i}\lra 0, \quad i\in\mathbb{Z}$$
with $P^{i}$ projective and $X^i\in {}^{\perp}\proja$ for all $i$ such that $X^0=X$. It follows from Lemma \ref{proposition-stfunctor-property-exseq} that there exist short exact sequences
$$0\lra \bar{F}(X^{i-1})\lra Q^i\lra \bar{F}(X^i)\lra 0, \quad i\in\mathbb{Z}$$
in $\mathcal{B}$ with $Q^i$ projective for all $i$.  Moreover, the objects $\bar{F}(X^i), i\in\mathbb{Z}$ are all in ${}^{\perp}\projb$ by (1). Hence, by Lemma \ref{lemma-gproj-ses}, the object $\bar{F}(X)$ is Gorenstein projective.
\end{proof}

It is well-known that, for an abelian category $\mathcal{A}$ with enough projective objects, there is a triangle embedding $\stGProj{\cal A}\hookrightarrow \Db{\mathcal{A}}/\Kb{\proja}$ induced by the canonical embedding $\mathcal{A}\hookrightarrow \Db{\mathcal{A}}$.  One may ask whether the stable functor is compatible with this embedding. The following theorem provides an affirmative answer.

\begin{Theo}
Let $\mathcal{A}$ and $\mathcal{B}$ be abelian categories with enough projective objects, and let $F:\Db{\mathcal{A}}\lra\Db{\mathcal{B}}$ be a  triangle functor. Then we have the following.

$(1)$. If $F$ is non-negative and  admits a right adjoint $G$ with $G(Q)\in\Kb{\proja}$ for all $Q\in\projb$, then  there 
 is a commutative diagram (up to natural isomorphism) of triangle functors.
$$\xymatrix@M=2mm@C=10mm{
\stGProj{\mathcal{A}}\ar[d]^{\bar{F}}\ar@{^{(}->}[r] &   \Db{\mathcal{A}}/\Kb{\proja}\ar[d]^{F}\\
\stGProj{\mathcal{B}}\ar@{^{(}->}[r]& \Db{\mathcal{B}}/\Kb{\projb},
}$$

$(2)$ If $F$ is  a uniformly bounded non-negative equivalence, then the functor $\bar{F}: \stGProj{\mathcal{A}}\lra \stGProj{\mathcal{B}}$ in the above diagram is a triangle equivalence.
\label{theorem-stable-gproj-equiv}
\end{Theo}
\begin{proof}  The commutative diagram follows from  Proposition \ref{proposition-stFun-commdiag} and Proposition \ref{proposition-gCM}.  It follows from Corollary \ref{corollary-comm-with-omega} and Proposition \ref{proposition-stfunctor-property-exseq} that   $\bar{F}: \stGProj{\mathcal{A}}\lra\stGProj{\mathcal{B}}$ is a triangle functor. 

(2)  Let $G$ be a quasi-inverse of $F$. Then $G$ is both a left adjoint and a right adjoint of $F$. By Lemma \ref{lemma-unibounded-non-negative}, there exists some integer $n>0$ such that $G[-n]$ is non-negative. Note that $F[n]$ is a right adjoint of $G[-n]$, and sends projective objects in $\mathcal{A}$ to complexes in $\Kb{\projb}$. By (1), the stable functor $\overline{G[-n]}$ of $G[-n]$ induces a triangle functor from $\stGProj{\mathcal{B}}$  to $\stGProj{\mathcal{A}}$.  Thus, by Theorem \ref{theorem-unique-F1F2} and Theorem  \ref{theorem-stFunctor-composite}, we have isomorphisms of functors
$$\bar{F}\circ \overline{G[-n]}\simeq \overline{F\circ G[-n]}\simeq \overline{[-n]}\simeq\Omega_{\mathcal{B}}^n \quad \mbox{ and }\quad \overline{G[-n]}\circ\bar{F}\simeq \overline{G[-n]\circ F}\simeq \overline{[-n]}\simeq \Omega_{\mathcal{A}}^n.$$
Note that $\Omega_{\mathcal{A}}$ and $\Omega_{\mathcal{B}}$ induce auto-equivalences of $\stGProj{\mathcal{A}}$ and $\stGProj{\mathcal{B}}$, respectively.  It follows that $\bar{F}: \stGProj{\mathcal{A}}\lra\stGProj{\mathcal{B}}$ is an equivalence.
\end{proof}

Let $F:\Db{\Modcat{A}}\ra\Db{\Modcat{B}}$ be a derived equivalence between two rings such that the tilting complex associated to $F$ has zero terms in all positive degrees.  By Lemma \ref{lemma-derequiv-non-negative},  the functor $F$ satisfies the   assumption of Theorem \ref{theorem-stable-gproj-equiv} (2). Thus, we have the following corollary.

\begin{Koro}
Let $A$ and $B$ be rings, and let $F: \Db{\Modcat{A}}\lra\Db{\Modcat{B}}$ be a non-negative derived equivalence. Then $\bar{F}: \stGProj{A}\lra\stGProj{B}$ is a triangle equivalence. 
\label{corollary-stgproj-equiv-ring-1}
\end{Koro}

For a given derived equivalence functor $F$ between two  rings,  by Lemma \ref{lemma-derequiv-non-negative}, $F[m]$ is non-negative when $m$ is sufficiently small.  The following corollary is then clear. 

\begin{Koro}
Let $A$ and $B$ be derived equivalent rings. Then $\stGProj{A}$ and $\stGProj{B}$ are triangle equivalent. 
\label{corollary-stgproj-equiv-ring-2}
\end{Koro}

Recall that a ring $A$ is called left coherent provided that the category $\modcat{A}$ of left finitely presented $A$-modules is an abelian category.  In this case, the finitely generated Gorenstein projective $A$-modules coincide with those Gorenstein projective modules in $\modcat{A}$. By $\Gproj{A}$ we denote the category of finitely generated Gorenstein projective $A$-modules, and by $\stGproj{A}$ we denote its stable category.

By Rickard's result in \cite{Rickard1989a}. For left coherent rings $A$ and $B$, $\Db{\modcat{A}}$ and $\Db{\modcat{B}}$ are triangle equivalent if and only if $\Db{\Modcat{A}}$ and $\Db{\Modcat{B}}$ are triangle equivalent, and every triangle equivalence between $\Db{\Modcat{A}}$ and $\Db{\Modcat{B}}$ restricts to a triangle equivalence between $\Db{\modcat{A}}$ and $\Db{\modcat{B}}$. Thus,  we obtain the following corollary.

\begin{Koro}
Let $A$ and $B$ be left coherent rings, and let $F: \Db{\modcat{A}}\ra \Db{\modcat{B}}$ be a non-negative derived equivalence. Then $\bar{F}: \stGproj{A}\ra \stGproj{B}$ are triangle equivalence. Particularly, the stable categories of finitely generated Gorenstein projective modules of two derived equivalent coherent rings are triangle equivalent. 
\label{corollary-stGproj-left-coherent}
\end{Koro}

\noindent
{\em Remark.} This generalizes a result of Kato \cite{Kato2002}, where it was proved that standard derived equivalences between two left and right coherent rings induce triangle equivalences between stable categories of finitely generated Gorenstein projective modules.

\section{An example}\label{section-example}
For a finite dimensional algebra $\Lambda$, in general, it is very hard  to find all the indecomposable Gorenstein projective modules in $\Gproj{\Lambda}$.  However, if $\Lambda$ is derived equivalent to another algebra $\Gamma$ for which the Gorenstein projective modules are known, then the stable functor will be helpful to describe the Gorenstein projective modules in $\Gproj{\Lambda}$.

Let $k$ be a field, and let $k[\epsilon]$ be the algebra of dual numbers, that is, the quotient algebra of the polynomial algebra $k[x]$ modulo the ideal generated by $x^2$. Let $A$ be the $k$-algebra given by the quiver
$$
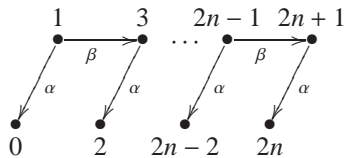
\xy <16pt,0pt>:
(0,0)*+0{\bullet}="1",
(2,0)*+0{\bullet}="3",
(4,0)*+0{\bullet}="2n-1",
(6,0)*+0{\bullet}="2n+1",  
(-1,-2)*+0{\bullet}="0",
(1, -2)*+0{\bullet}="2", 
(3,-2)*+0{\bullet}="2n-2",
(5,  -2)*+0{\bullet}="2n", 
(3,-.1)*+0{\cdots},
(0, .55)*+{1},
(2, .55)*+{3}, 
(4, .55)*+{2n-1}, 
(6, .55)*+{2n+1}, 
(-1, -2.5)*+{0}, 
(1,-2.5)*+{2}, 
(3, -2.5)*+{2n-2}, 
(5, -2.5)*+{2n},
{\ar^{\alpha} "1"; "0"},
{\ar_(.4){\beta} "1"; "3"},
{\ar^{\alpha} "3"; "2"}, 
{\ar_(.4){\beta} "2n-1"; "2n+1"},
{\ar^{\alpha}  "2n-1"; "2n-2"}, 
{\ar^{\alpha}  "2n+1"; "2n"}, 
\endxy$$
with all possible relations  $\beta\alpha=0$. Then $A$ is a tilted algebra, and is derived equivalent to the path algebra, denoted by $B$,  of the following quiver.
$$
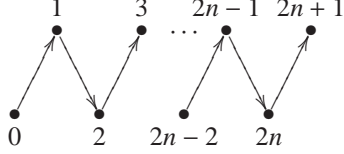
\xy <16pt,0pt>:
(0,0)*+0{\bullet}="1",
(2,0)*+0{\bullet}="3",
(4,0)*+0{\bullet}="2n-1",
(6,0)*+0{\bullet}="2n+1",  
(-1,-2)*+0{\bullet}="0",
(1, -2)*+0{\bullet}="2", 
(3,-2)*+0{\bullet}="2n-2",
(5,  -2)*+0{\bullet}="2n", 
(3,-.1)*+0{\cdots},
(0, .55)*+{1},
(2, .55)*+{3}, 
(4, .55)*+{2n-1}, 
(6, .55)*+{2n+1}, 
(-1, -2.5)*+{0}, 
(1,-2.5)*+{2}, 
(3, -2.5)*+{2n-2}, 
(5, -2.5)*+{2n},
{\ar "0"; "1"},
{\ar "1"; "2"},
{\ar "2"; "3"}, 
{\ar "2n-2"; "2n-1"}, 
{\ar  "2n-1"; "2n"},
{\ar  "2n"; "2n+1"},  
\endxy$$
Let $\Lambda:=k[\epsilon]\otimes_kA$ and let $\Gamma:=k[\epsilon]\otimes_kB$. Then $\Lambda$ and $\Gamma$ are also derived equivalent.  The Gorenstein projective modules over $\Gamma$ have been described by Ringel and Zhang in \cite{Ringel2011-preprint}. They also proved that $\stGproj{\Gamma}$ is equivalent to the orbit category $\Db{B}/[1]$.  This means that the indecomposable non-projective Gorenstein projective $\Gamma$-modules are one-to-one correspondent to the indecomposable $B$-modules. Then correspondence reads as follows. Let $S$ be the unique simple $k[\epsilon]$-module, and let $Q_i$ be the indecomposable projective $B$-module corresponding to the vertex $i$ for all $i$. For each $i\in\{0, 1,\cdots, 2n+1\}$, and $1\leq l\leq 2n+2-i$, we denote by $X(i, l)$ the indecomposable $B$-module with top vertex $i$ and length $l$. We write $M(i, l)$ for the corresponding Gorenstein projective $\Gamma$-module.  If $X(i, l)$ is projective, that is, $i+l=2n+2$, then $M(i, l)=S\otimes Q_i$. If $X(i, l)$ is not projective, then there is an short exact sequence
$$0\lra S\otimes Q_{i+l}\lra M(i, l)\lra S\otimes Q_i\lra 0. \quad\quad (*)$$
Here we shall use the stable functor of the derived equivalence between $\Gamma$ and $\Lambda$ to get all the indecomposable Gorenstein projective modules over $\Lambda$.

  For each $i\in\{0,1, \cdots, 2n+1\}$, we denote by $P_i$  the indecomposable projective $A$-module corresponding to the vertex $i$.  The derived equivalence between $A$ and $B$ is given by the tilting module
$$\bigoplus_{i=0}^n\big(P_{2i+1}\oplus\tau^{-1}S_{2i}\big), $$
where $S_{2i}$ is the simple $A$-module corresponding to the vertex $2i$. Note that $\tau^{-1}S_{2i}$ has a projective resolution
$$0\lra P_{2i}\lra P_{2i+1}\lra \tau^{-1}S_{2i}\lra 0.$$
Thus, we get a derived equivalence $F: \Db{B}\lra\Db{A}$ such that $F(Q_{2i+1})\simeq P_{2i+1}[-1]$ and $F(Q_{2i})$ is 
$$0\lra P_{2i}\lra P_{2i+1}\lra 0$$
with $P_{2i}$ in degree zero for all $0\leq i\leq n$. By \cite{Rickard1991}, there is a derived equivalence $F': \Db{\Gamma}\lra \Db{\Lambda}$, which sends $k[\epsilon]\otimes Q_{2i+1}$ to $k[\epsilon]\otimes P_{2i+1}[-1]$, and sends $k[\epsilon]\otimes Q_{2i}$ to  the complex 
$$0\lra k[\epsilon]\otimes P_{2i}\lra k[\epsilon]\otimes P_{2i+1}\lra 0$$
for all $0\leq i\leq n$. 

\smallskip
For each $i\in\{0, 1, \cdots, 2n+1\}$, and for each $1\leq l\leq 2n+2-i$, let $N(i, l)$ be the image of $M(i, l)$ under the stable functor of $F'$. Then it is easy to see that  the image $N(2i+1, 2n-2i+1)$ of $M(2i+1, 2n-2i+1) (=S\otimes Q_{2i+1})$ is $\Omega (S\otimes P_{2i+1})$, which is isomorphic to $S\otimes P_{2i+1}$.  The module $N(2i, 2n-2i+2)$ fits into the following  pullback diagram
$$\xymatrix{
N(2i, 2n-2i+2)\ar[r]\ar[d] & k[\epsilon]\otimes P_{2i+1}\ar@{->>}[d]\\
S\otimes P_{2i}\ar[r] & S\otimes P_{2i+1},
}$$
and can be diagrammatically  presented as follows
$$
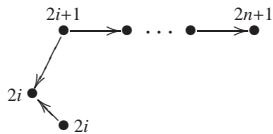
\xy <12pt,0pt>:
(0,0)*+0{\bullet}="2i+1",
(2,0)*+0{\bullet}="2i+3",
(4,0)*+0{\bullet}="2n-1",
(6,0)*+0{\bullet}="2n+1",  
(-1,-2)*+0{\bullet}="2i",
(0,-3)*+0{\bullet}="2ia", 
(3,-.1)*+0{\cdots},
(0, .55)*+{\scriptstyle{2i+1}},  
(6, .55)*+{\scriptstyle{2n+1}}, 
(-1.55, -2)*+{\scriptstyle{2i}},
(.55, -3)*+{\scriptstyle{2i}}, 
{\ar "2i+1"; "2i"},
{\ar "2i+1"; "2i+3"}, 
{\ar "2n-1"; "2n+1"},
{\ar "2ia"; "2i"},  
\endxy$$
Each vertex of the above diagram corresponds to a basis vector of the module, and the arrow from $2i$ to $2i$ corresponds to the action of $\epsilon$. The other arrow corresponds to the action of the corresponding arrow in the quiver of $A$.  Since the stable functor is a triangle equivalence, we can use the short exact sequence $(*)$ to get $N(i, l)$ for $1\leq l<2n+2-i$. The result can be listed as the following table. 
\begin{center}
\begin{tabular}{|c|c|}
\hline
$i, l$ &  $N(i, l)$ \\
\hline
$i$ and $l$ are even &  $\xy <12pt,0pt>:
(0,0)*+0{\bullet}="i+1",
(2,0)*+0{\bullet}="i+3",
(4,0)*+0{\bullet}="i+l+1",
(4,.55)*+{\scriptstyle{i+l+1}}, 
(6,0)*+0{\bullet}="i+l+3",
(8, 0)*+0{\bullet}="2n-1",
(10, 0)*+0{\bullet}="2n+1",   
(-1,-2)*+0{\bullet}="i",
(0,-3)*+0{\bullet}="ia", 
(3,-.1)*+0{\cdots},
(7,-.1)*+0{\cdots},
(5,-1)*+0{\bullet}="i+l+1a",
(7,-1)*+0{\bullet}="i+l+3a",
(9, -1)*+0{\bullet}="2n-1a",
(11, -1)*+0{\bullet}="2n+1a", 
(4,-3)*+0{\bullet}="i+la",
(5, -4)*+0{\bullet}="i+lb",
(8,-1.1)*+0{\cdots},
(0, .55)*+{\scriptstyle{i+1}},  
(10, .55)*+{\scriptstyle{2n+1}}, 
(-1.4, -2)*+{\scriptstyle{i}},
(.4, -3)*+{\scriptstyle{i}}, 
(3.9, -1)*+{\scriptstyle{i+l+1}}, 
(3.2, -3)*+{\scriptstyle{i+l}}, 
(5.8, -4)*+{\scriptstyle{i+l}}, 
{\ar "i+1"; "i"},
{\ar "i+1"; "i+3"}, 
{\ar "2n-1"; "2n+1"},
{\ar "i+l+1"; "i+l+3"}, 
{\ar "ia"; "i"},  
{\ar "i+l+1a"; "i+l+3a"}, 
{\ar "2n-1a"; "2n+1a"},
{\ar "i+l+1a"; "i+la"}, 
{\ar "i+lb"; "i+la"}, 
{\ar "i+l+1a"; "i+l+1"}, 
{\ar "i+l+3a"; "i+l+3"},
{\ar "2n-1a"; "2n-1"},  
{\ar "2n+1a"; "2n+1"}, 
\endxy$   \\
\hline
$i$ is even, $l>1$ is odd & 
$\xy <12pt,0pt>:
(0,0)*+0{\bullet}="i+1",
(2,0)*+0{\bullet}="i+3",
(4,0)*+0{\bullet}="i+l+1",
(4,.55)*+ {\scriptstyle i+l}, 
(6,0)*+0{\bullet}="i+l+3",
(8, 0)*+0{\bullet}="2n-1",
(10, 0)*+0{\bullet}="2n+1",   
(-1,-2)*+0{\bullet}="i",
(0,-3)*+0{\bullet}="ia", 
(3,-.1)*+0{\cdots},
(7,-.1)*+0{\cdots},
(5,-1)*+0{\bullet}="i+l+1a",
(7,-1)*+0{\bullet}="i+l+3a",
(9, -1)*+0{\bullet}="2n-1a",
(11, -1)*+0{\bullet}="2n+1a", 
(4,-3)*+0{\bullet}="i+la",
(8,-1.1)*+0{\cdots},
(0, .55)*+{\scriptstyle i+1},  
(10, .55)*+{\scriptstyle 2n+1}, 
(-1.4, -2)*+{\scriptstyle i},
(.4, -3)*+{\scriptstyle i}, 
(4.2, -1)*+{\scriptstyle i+l}, 
(5.1, -3)*+{\scriptstyle i+l-1}, 
{\ar "i+1"; "i"},
{\ar "i+1"; "i+3"}, 
{\ar "2n-1"; "2n+1"},
{\ar "i+l+1"; "i+l+3"}, 
{\ar "ia"; "i"},  
{\ar "i+l+1a"; "i+l+3a"}, 
{\ar "2n-1a"; "2n+1a"},
{\ar "i+l+1a"; "i+la"}, 
{\ar "i+l+1a"; "i+l+1"}, 
{\ar "i+l+3a"; "i+l+3"},
{\ar "2n-1a"; "2n-1"},  
{\ar "2n+1a"; "2n+1"}, 
\endxy$\\
\hline
$i$ is even, $l=1$ & $\xy <12pt,0pt>:
(0,0)*+0{\bullet}="i",
(.55, 0)*+{\scriptstyle i}, 
\endxy$\\

\hline
$i$ is odd, $l$ is even & 
$\xy <12pt,0pt>:
(0,0)*+0{\bullet}="i",
(2,0)*+0{\bullet}="i+2",
(4,0)*+0{\bullet}="i+l",
(4,.55)*+{\scriptstyle i+l}, 
(6,0)*+0{\bullet}="i+l+2",
(8, 0)*+0{\bullet}="2n-1",
(10, 0)*+0{\bullet}="2n+1",   
(-1,-2)*+0{\bullet}="i-1",
(3,-.1)*+0{\cdots},
(7,-.1)*+0{\cdots},
(5,-1)*+0{\bullet}="i+la",
(7,-1)*+0{\bullet}="i+l+2a",
(9, -1)*+0{\bullet}="2n-1a",
(11, -1)*+0{\bullet}="2n+1a", 
(4,-3)*+0{\bullet}="i+l-1a",
(8,-1.1)*+0{\cdots},
(0, .55)*+{\scriptstyle i},  
(10, .55)*+{\scriptstyle 2n+1}, 
(-.2, -2)*+{\scriptstyle i-1}, 
(4.1, -1)*+{\scriptstyle i+l}, 
(5.1, -3)*+{\scriptstyle i+l-1}, 
{\ar "i"; "i-1"},
{\ar "i"; "i+2"}, 
{\ar "2n-1"; "2n+1"},
{\ar "i+l"; "i+l+2"}, 
{\ar "i+la"; "i+l+2a"}, 
{\ar "2n-1a"; "2n+1a"},
{\ar "i+la"; "i+l-1a"}, 
{\ar "i+la"; "i+l"}, 
{\ar "i+l+2a"; "i+l+2"},
{\ar "2n-1a"; "2n-1"},  
{\ar "2n+1a"; "2n+1"}, 
\endxy$\\
\hline 

$i$ is odd, $l$ is odd & 

$\xy <12pt,0pt>:
(0,0)*+0{\bullet}="i",
(2,0)*+0{\bullet}="i+2",
(4,0)*+0{\bullet}="i+l+1",
(4,.55)*+{\scriptstyle i+l+1}, 
(6,0)*+0{\bullet}="i+l+3",
(8, 0)*+0{\bullet}="2n-1",
(10, 0)*+0{\bullet}="2n+1",   
(-1,-2)*+0{\bullet}="i-1",
(3,-.1)*+0{\cdots},
(7,-.1)*+0{\cdots},
(5,-1)*+0{\bullet}="i+l+1a",
(7,-1)*+0{\bullet}="i+l+3a",
(9, -1)*+0{\bullet}="2n-1a",
(11, -1)*+0{\bullet}="2n+1a", 
(4,-3)*+0{\bullet}="i+la",
(5, -4)*+0{\bullet}="i+lb", 
(8,-1.1)*+0{\cdots},
(0, .55)*+{\scriptstyle i},  
(10, .55)*+{\scriptstyle 2n+1}, 
(-.2, -2)*+{\scriptstyle i-1}, 
(3.8, -1)*+{\scriptstyle i+l+1}, 
(4.8, -3)*+{\scriptstyle i+l},
(5.8, -4)*+{\scriptstyle{i+l}}, 
{\ar "i"; "i-1"},
{\ar "i"; "i+2"}, 
{\ar "2n-1"; "2n+1"},
{\ar "i+l+1"; "i+l+3"}, 
{\ar "i+l+1a"; "i+l+3a"}, 
{\ar "2n-1a"; "2n+1a"},
{\ar "i+l+1a"; "i+la"}, 
{\ar "i+l+1a"; "i+l+1"}, 
{\ar "i+l+3a"; "i+l+3"},
{\ar "2n-1a"; "2n-1"},  
{\ar "2n+1a"; "2n+1"}, 
{\ar "i+lb"; "i+la"}, 
\endxy$\\
\hline 

\end{tabular}
\end{center}
In case that $n=1$,  the algebra $A$ is given by the quiver
$$
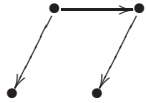
\xy <16pt,0pt>:
(0,0)*+0{\bullet}="1",
(2,0)*+0{\bullet}="3",
(-1, -2)*+0{\bullet}="0", 
(1, -2)*+0{\bullet}="2", 
{\ar "1"; "0"}, 
{\ar "1"; "3"}, 
{\ar "3"; "2"}, 
\endxy$$

\noindent
The Auslander-Reiten quiver of $\stGproj{\Lambda}$ can be drawn as follows. 

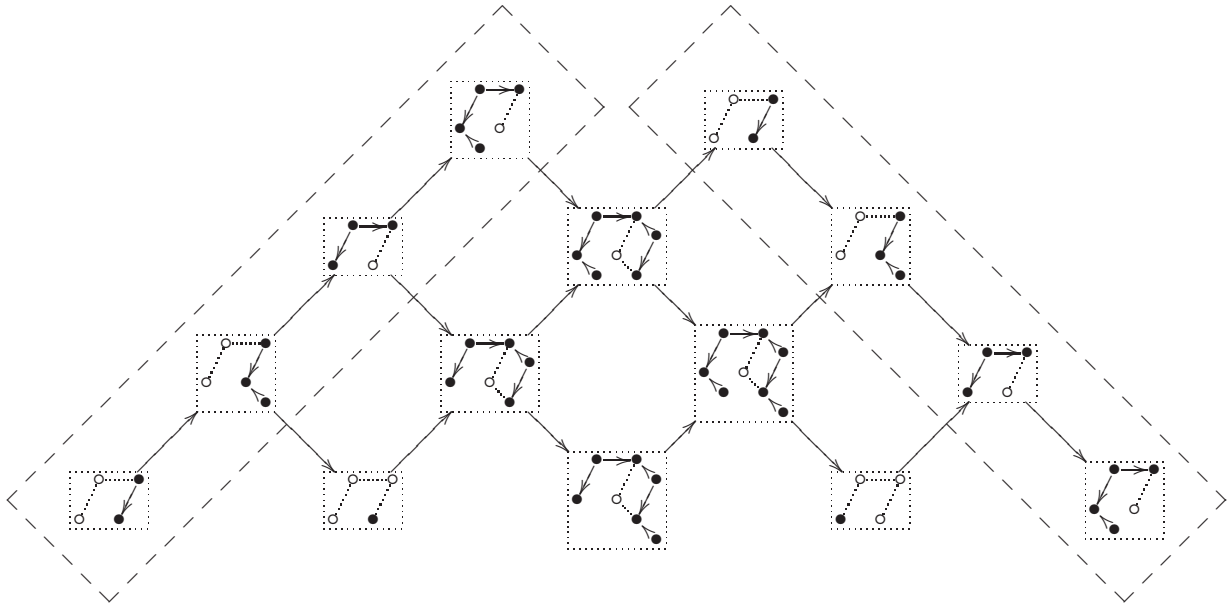
\newcommand{\maa}{\xy <7.5pt,0pt>:
(0,0)*+0{\circ}="1",
(2,0)*+0{\bullet}="3",
(-1,-2)*+0{\circ}="0",
(1, -2)*+0{\bullet}="2",
{\ar@{..} "1"; "0"},
{\ar@{..} "1"; "3"},
{\ar "3"; "2"},
\endxy}

\newcommand{\mab}{\xy <7.5pt,0pt>:
(0,0)*+0{\circ}="1",
(2,0)*+0{\bullet}="3",
(-1,-2)*+0{\circ}="0",
(1, -2)*+0{\bullet}="2",
(2, -3)*+0{\bullet}="2a",
{\ar@{..} "1"; "0"},
{\ar@{..} "1"; "3"},
{\ar "3"; "2"},
{\ar "2a"; "2"},
\endxy}

\newcommand{\mac}{\xy <7.5pt,0pt>:
(0,0)*+0{\bullet}="1",
(2,0)*+0{\bullet}="3",
(-1,-2)*+0{\bullet}="0",
(1, -2)*+0{\circ}="2",
{\ar "1"; "0"},
{\ar "1"; "3"},
{\ar@{..} "3"; "2"},
\endxy}

\newcommand{\mad}{\xy <7.5pt,0pt>:
(0,0)*+0{\bullet}="1",
(2,0)*+0{\bullet}="3",
(-1,-2)*+0{\bullet}="0",
(1, -2)*+0{\circ}="2",
(0, -3)*+0{\bullet}="0a",
{\ar "1"; "0"},
{\ar "1"; "3"},
{\ar@{..} "3"; "2"},
{\ar "0a"; "0"},
\endxy}

\newcommand{\mba}{\xy <7.5pt,0pt>:
(0,0)*+0{\circ}="1",
(2,0)*+0{\circ}="3",
(-1,-2)*+0{\circ}="0",
(1, -2)*+0{\bullet}="2",  ,
{\ar@{..} "1"; "0"},
{\ar@{..} "1"; "3"},
{\ar@{..} "3"; "2"},
\endxy}

\newcommand{\mbb}{\xy <7.5pt,0pt>:
(0,0)*+0{\bullet}="1",
(2,0)*+0{\bullet}="3",
(-1,-2)*+0{\bullet}="0",
(1, -2)*+0{\circ}="2",
(2, -3)*+0{\bullet}="2a",
(3, -1)*+0{\bullet}="3a",
{\ar "1"; "0"},
{\ar "1"; "3"},
{\ar@{..} "3"; "2"},
{\ar "3a"; "3"},
{\ar@{..} "2a"; "2"},
{\ar "3a"; "2a"},
\endxy}

\newcommand{\mbc}{\xy <7.5pt,0pt>:
(0,0)*+0{\bullet}="1",
(2,0)*+0{\bullet}="3",
(-1,-2)*+0{\bullet}="0",
(1, -2)*+0{\circ}="2",
(0, -3)*+0{\bullet}="0a",
(2, -3)*+0{\bullet}="2a",
(3, -1)*+0{\bullet}="3a",
{\ar "1"; "0"},
{\ar "1"; "3"},
{\ar@{..} "3"; "2"},
{\ar "3a"; "3"},
{\ar@{..} "2a"; "2"},
{\ar "0a"; "0"},
{\ar "3a"; "2a"},
\endxy}

\newcommand{\mca}{\xy <7.5pt,0pt>:
(0,0)*+0{\bullet}="1",
(2,0)*+0{\bullet}="3",
(-1,-2)*+0{\bullet}="0",
(1, -2)*+0{\circ}="2",
(2, -3)*+0{\bullet}="2a",
(3, -4)*+0{\bullet}="2b",
(3, -1)*+0{\bullet}="3a",
{\ar "1"; "0"},
{\ar "1"; "3"},
{\ar@{..} "3"; "2"},
{\ar "3a"; "3"},
{\ar@{..} "2a"; "2"},
{\ar "2b"; "2a"},
{\ar "3a"; "2a"},
\endxy}

\newcommand{\mcb}{\xy <7.5pt,0pt>:
(0,0)*+0{\bullet}="1",
(2,0)*+0{\bullet}="3",
(-1,-2)*+0{\bullet}="0",
(1, -2)*+0{\circ}="2",
(0, -3)*+0{\bullet}="0a",
(2, -3)*+0{\bullet}="2a",
(3, -4)*+0{\bullet}="2b",
(3, -1)*+0{\bullet}="3a",
{\ar "1"; "0"},
{\ar "1"; "3"},
{\ar@{..} "3"; "2"},
{\ar "3a"; "3"},
{\ar@{..} "2a"; "2"},
{\ar "0a"; "0"},
{\ar "2b"; "2a"},
{\ar "3a"; "2a"},
\endxy}

\newcommand{\mda}{\xy <7.5pt,0pt>:
(0,0)*+0{\circ}="1",
(2,0)*+0{\circ}="3",
(-1,-2)*+0{\bullet}="0",
(1, -2)*+0{\circ}="2",
{\ar@{..} "1"; "0"},
{\ar@{..} "1"; "3"},
{\ar@{..} "3"; "2"},
\endxy}

$$\xy<48pt, 0pt>:
(0, 0)*+{\maa}*\frm{.}="maa",
(1, 1)*+{\mab}*\frm{.}="mab",
(2,2)*+{\mac}*\frm{.}="mac",
(3,3)*+{\mad}*\frm{.}="mad",
(2,0)*+{\mba}*\frm{.}="mba",
(3,1)*+{\mbb}*\frm{.}="mbb",
(4,2)*+{\mbc}*\frm{.}="mbc",
(5,3)*+{\maa}*\frm{.}="maa1",
(4, 0)*+{\mca}*\frm{.}="mca",
(5, 1)*+{\mcb}*\frm{.}="mcb",
(6,2)*+{\mab}*\frm{.}="mab1",
(6,0)*+{\mda}*\frm{.}="mda",
(7, 1)*+{\mac}*\frm{.}="mac1",
(8, 0)*+{\mad}*\frm{.}="mad1",
{\ar "maa"; "mab"},
{\ar "mab"; "mac"},
{\ar "mac"; "mad"},
{\ar "mba"; "mbb"},
{\ar "mab"; "mba"},
{\ar "mac"; "mbb"},
{\ar "mbb"; "mbc"},
{\ar "mad"; "mbc"},
{\ar "mbc"; "maa1"},
{\ar "mbb"; "mca"},
{\ar "mca"; "mcb"},
{\ar "mcb"; "mab1"},
{\ar "mda"; "mac1"},
{\ar "mbc"; "mcb"},
{\ar "mcb"; "mda"},
{\ar "maa1"; "mab1"},
{\ar "mab1"; "mac1"},
{\ar "mac1"; "mad1"},
@={(-0.8,0), (3.1, 3.9), (3.9, 3.1), (0, -0.8)}, 
s0="prev" @@{; "prev"; **@{--}="prev"} , 
@={(8.8,0), (4.9, 3.9), (4.1,3.1), (8, -0.8)}, 
s0="prev" @@{; "prev"; **@{--}="prev"} , 
\endxy$$
The modules in the two dashed frames are identified correspondingly.

\section{Concluding remarks}\label{section-remarks}
Our results can be applied abelian categories with enough injective objects (e.g. Grothendieck categories). One just need to consider their opposite categories, which are abelian categories with enough projective objects.  Our results can also be used to give shorter proofs of some known results on homological conjectures. 

\medskip
In the following, we assume that $A$ and $B$ are derived equivalent left coherent rings, and  $F:\Db{\modcat{A}}\ra\Db{\modcat{B}}$ is a derived equivalence. Without loss of generality, we can assume that $F$ is non-negative and the tilting complex associated to $F$ has terms only in degrees $0, \cdots, n$. Let $G$ be a quasi-inverse of $F$. Then $G[-n]$ is also non-negative.

\medskip
\noindent
{\em Finitistic dimension.} The finitistic dimension of a left coherent ring  is the supremum of projective dimensions of finitely presented modules with finite projective dimensions. The finiteness of finitistic dimension is proved to be preserved under derived equivalences in \cite{Pan2009a}. With the stable functor, the proof will be very easy.   We claim that $|\findim(A)-\findim(B)|\leq n$, where $\findim$ stands for the finitistic dimension.  To prove this, it is sufficient to prove that,  for each $A$-module $X$, there are inequalities between the projective dimensions of $X$ and $\bar{F}(X)$: 
$$\projdim {}_B\bar{F}(X)\leq \projdim {}_AX\leq \projdim {}_B\bar{F}(X)+n.$$ 
We first prove the first inequality. Suppose that $\projdim{}_AX=m$. Then $\Omega_A^m(X)\simeq 0$ in $\stmodcat{A}$, and consequently $\Omega_B^m\circ \bar{F}(X)\simeq \bar{F}\circ \Omega_A^m(X)\simeq 0$ in $\stmodcat{B}$, where the first isomorphism follows from Corollary \ref{corollary-comm-with-omega}.  Hence
 $$\projdim {}_B\bar{F}(X)\leq m=\projdim{}_AX.$$ 
The proof of the second inequality goes as follows.  Suppose that $\projdim{}_B\bar{F}(X)=m$. Then we have the following isomorphisms in $\stmodcat{A}$: 
$$\Omega_A^{m+n}(X)\simeq \overline{[-n-m]}(X)\simeq \overline{G[-n]\circ [-m]\circ  F}(X)\simeq \overline{G[-n]}\circ\Omega_B^m\circ\bar{F}(X)\simeq 0, $$
where the third isomorphism follows from  Theorem \ref{theorem-stFunctor-composite}. This implies that $$\projdim {}_AX\leq m+n=\projdim{}_B\bar{F}(X)+n.$$

\noindent
{\em Syzygy finiteness.}
A left coherent ring $\Lambda$ is called   $\Omega^m$-finite provided that $\add(\Omega_{\Lambda}^m(\modcat{\Lambda}))$ contains only finitely many isomorphism classes of  indecomposable  $\Lambda$-modules, and is called {\em syzygy-finite} if $A$ is $\Omega^m$-finite for some $m$. Clearly,  a syzygy-finite algebra always has finite finitistic dimension. With the help of the stable functor, we can prove that: 

\smallskip
{\em If $A$ is $\Omega^m$-finite, then $B$ is $\Omega^{m+n}$-finite. In particular $A$ is syzygy-finite if and only if so is $B$.}

\smallskip
\noindent The proof of the above statement is almost trivial. Let $X$ be a $B$-module. By assumption, there is an $A$-module $M$ such that $\Omega_A^m\bar{F}(X)\in \add(M)$. Applying the stable functor of $G[-n]$, we see that $\Omega_B^{m+n}(X)$, which is isomorphic to $\overline{G\circ[-n]}\circ \Omega_A^m\circ \bar{F}(X)$ in $\stmodcat{B}$, is in $\add(B\oplus \overline{G[-n]}(M))$, showing that $B$ is $\Omega^{m+n}$-finite.  

\medskip
\noindent 
{\em Generalized Auslander-Reiten conjecture.} This conjecture says that a module $X$ over an Artin algebra $\Lambda$ satisfying $\Ext_{\Lambda}^i(X, X\oplus\Lambda)=0$ for all $i>m\geq 0$ has projective dimension $\leq m$. Via the stable functor, the second author proved in \cite{Pan2013c} that $A$ satisfies the generalized Auslander-Reiten conjecture if and only so does $B$. This was also proved by Wei \cite{Wei2012} and by Diveris and Purin \cite{Diveris2012} independently.

\bigskip
\noindent{\bf Acknowledgements.}
The research work of both authors are partially supported by NSFC and the Fundamental Research Funds for the Central Universities. W. Hu thanks BNSF(1132005, KZ201410028033) for partial support, and  S.Y. Pan is also partially supported by  the Scientific Research Foundation for the Returned Overseas Chinese Scholars, State Education Ministry.

%
%
%

\bibliographystyle{aomalpha}
\bibliography{../refData}

\bigskip
Wei Hu

\medskip
School of Mathematical Sciences, Laboratory of Mathematics and Complex Systems, MOE, Beijing Normal
University, 100875 Beijing, China

\smallskip
Beijing Center for Mathematics and Information Interdisciplinary Sciences, 100048 Beijing, China

\medskip
{\tt Email: huwei@bnu.edu.cn}

\bigskip
Shengyong Pan

\medskip
Department of Mathematics, Beijing Jiaotong
University, 100044 Beijing, China

\smallskip
Beijing Center for Mathematics and Information Interdisciplinary Sciences, 100048 Beijing, China

\medskip
{\tt Email: shypan@bjtu.edu.cn}
\end{document}